\documentclass[12pt,oneside,a4papre]{amsart}

\usepackage{amssymb}
\usepackage{amsmath}
\usepackage{amsthm}
\usepackage{amscd}

\usepackage{longtable}
\usepackage{mathrsfs}

\usepackage[dvipdfmx]{xcolor}
\usepackage[dvipdfmx]{pict2e}
\usepackage[dvipdfmx]{graphicx}
\usepackage{comment}

\usepackage[scr=boondoxo]{mathalfa}

\usepackage{tikz}

\usepackage[all]{xy}

\usetikzlibrary{cd}

\theoremstyle{plain}
\newtheorem{theorem}{Theorem}[section]
\newtheorem{lemma}[theorem]{Lemma}
\newtheorem{corollary}[theorem]{Corollary}
\newtheorem{proposition}[theorem]{Proposition}

\theoremstyle{definition}
\newtheorem{definition}[theorem]{Definition}

\newcommand{\kah}{K\"{a}hler }

\newcommand{\idd}{i\partial\overline{\partial}}

\subjclass[2020]{32L05, 41A36}
\keywords{ 
$L^2$-estimates, singular Hermitian metrics, Nakano positivity.}

\begin{document}
\title
[Curvature operator and $L^2$-estimate condition]
{Curvature operator of holomorphic vector bundles and $L^2$-estimate condition for $(n,q)$ and $(p,n)$-forms}
\author{Yuta Watanabe}
\date{}

\begin{abstract}
    We study the positivity properties of the curvature operator for holomorphic Hermitian vector bundles. 
    The characterization of Nakano semi-positivity by $L^2$-estimate is already known. Applying our results, we give new characterizations of Nakano semi-negativity.
\end{abstract}

\vspace{-5mm}

\maketitle

\tableofcontents

\vspace{-8mm}

\section{Introduction}\label{section:1}

The aim of the present paper is to study the relation between the positivity properties of curvature operators for holomorphic Hermitian vector bundles and $L^2$-estimates by using the $(p,n)$-$L^2$-estimate condition (see Definition \ref{def (p,n)-condition}).
This type of condition was firstly introduced in \cite{HI20}, which was named as the twisted H\"{o}rmander condition.
After that, in a paper \cite{DNWZ20}, Deng et al. generalized this concept as the optimal $L^p$-estimate condition 
and gave a new characterizations of Nakano semi-positivity by $L^2$-estimate. 
Here, Nakano positive is equivalent to the positivity of curvature operator for $(n,1)$-forms.
In this paper, by examining the properties of curvature operators, we extend this characterization in \cite{DNWZ20} from $(n,1)$-forms to $(n,q)$ and $(p,n)$-forms
and obtain a characterizations of Nakano semi-negativity by $L^2$-estimates.
Finally we obtain one definition of Nakano semi-negativity and dual Nakano semi-positivity for singular Hermitian metrics with $L^2$-estimates.

Let $(X,\omega)$ be a complex manifold of complex dimension $n$ equipped with a Hermitian metric $\omega$ and $(E,h)$ be a holomorphic Hermitian vector bundle of rank $r$ over $X$.
Let $D=D'+\overline{\partial}$ be the Chern connection of $(E,h)$, and $\Theta_{E,h}=[D',\overline{\partial}]=D'\overline{\partial}+\overline{\partial}D'$ be the Chern curvature tensor. 
Let $(U,(z_1,\cdots,z_n))$ be local coordinates. Denote by $(e_1,\cdots,e_r)$ an orthonormal frame of $E$ over $U\subset X$, and
\begin{align*}
    i\Theta_{E,h,x_0}=i\sum_{j,k,\lambda,\mu}c_{jk\lambda\mu}dz_j\wedge d\overline{z}_k\otimes e^*_\lambda\otimes e_\mu,~\,\, \overline{c}_{jk\lambda\mu}=c_{kj\mu\lambda}.
\end{align*}
To $i\Theta_{E,h}$ corresponds a natural Hermitian form $\theta_{E,h}$ on $T_X\otimes E$ defined by 
\begin{align*}
    \theta_{E,h}(u,u)&=\sum c_{jk\lambda\mu}u_{j\lambda}\overline{u}_{k\mu},~\,\, u=\sum u_{j\lambda}\frac{\partial}{\partial z_j}\otimes e_\lambda\in T_{X,x_0}\otimes E_x,\\
    \mathrm{i.e.}~\,\,\, \theta_{E,h}&=\sum c_{jk\lambda\mu}(dz_j\otimes e^*_\lambda)\otimes\overline{(dz_k\otimes e^*_\mu)}.
\end{align*}

\begin{definition}
    Let $X$ be a complex manifold and $(E,h)$ be a holomorphic Hermitian vector bundle over $X$.
    \begin{itemize}
        \item $(E,h)$ is said to be $\it{Nakano ~positive}$ (resp. $\it{Nakano ~semi}$-$\it{positive}$) if $\theta_{E,h}$ is positive (resp. semi-positive) definite as a Hermitian form on $T_X\otimes E$, 
                    i.e. for any $u\in T_X\otimes E$, $u\ne0$, we have \[ \theta_{E,h}(u,u)>0~\,\,(\mathrm{resp}. \geq0). \]
                    We write $(E,h)>_{Nak}0$, i.e. $i\Theta_{E,h}>_{Nak}0$ (resp. $\geq_{Nak}0$) for Nakano positivity (resp. semi-negativity).
        \item $(E,h)$ is said to be $\it{Griffiths ~positive}$ (resp. $\it{Griffiths ~semi}$-$\it{positive}$) if 
                    for any $\xi\in T_{X,x}$, $\xi\ne0$ and $s\in E_x$, $s\ne0$, we have \[ \theta_{E,h}(\xi\otimes s,\xi\otimes s)>0~\,\,(\mathrm{resp}. \geq0). \]
                    We write $(E,h)>_{Grif}0$, i.e. $i\Theta_{E,h}>_{Grif}0$ (resp. $\geq_{Grif}0$) for Griffiths positivity (resp. semi-negativity).
        \item $\it{Nakano ~negative}$ (resp. $\it{Nakano ~semi}$-$\it{negative}$) and $\it{Griffiths ~negative}$ (resp. $\it{Griffiths ~semi}$-$\it{negative}$) are similarly defined by replacing $>0$ (resp. $\geq0$) by $<0$ (resp. $\leq0$) in the above definitions respectively.
    \end{itemize}
\end{definition}

We now explain a few notions to state our results more precisely. Let $\mathcal{E}^{p,q}(E)$ be the sheaf of germs of $\mathcal{C}^\infty$ sections of $\Lambda^{p,q}T^*_X\otimes E$ and $\mathcal{D}^{p,q}(E)$ be the space of $\mathcal{C}^\infty$ sections of $\Lambda^{p,q}T^*_X\otimes E$ with compact support on $X$.
We say that $\{(U_\alpha,\iota_\alpha)\}_\alpha$ is a $\it{local ~Stain ~coordinate ~system}$ if any local coordinates $\iota_\alpha:U_\alpha\to \iota(U_\alpha)\subset\mathbb{C}^n$ satisfies that $\iota(U_\alpha)$ is Stein. 
By definition, every complex manifold always has a local Stein coordinate system.

\begin{definition}$(\mathrm{cf.~}$\cite{BP08},~\cite{Rau15} and \cite{PT18})
    Let $X$ be a complex manifold and $E$ be a holomorphic vector bundle over $X$. We say that $h$ is a $\it{singular~Hermitian~metric}$ on $E$ if $h$ is a measurable map from the base manifold $X$ to the space of non-negative Hermitian forms on the fibers satisfying $0<\mathrm{det} h<+\infty$ almost everywhere.
\end{definition}

We already know that for a singular Hermitian metric, we cannot always define the curvature currents with measure coefficients (see \cite{Rau15}). 
However, the following definition can be defined with the singular case by not using the curvature currents of a singular Hermitian metric directly.

\begin{definition}\label{def (n,q)-condition}
    Let $(X,\omega)$ be a \kah manifold of dimension $n$ which admits a positive holomorphic Hermitian line bundle and $E$ be a holomorphic vector bundle over $X$ equipped with a (singular) Hermitian metric $h$.
    $(E,h)$ satisfies $\it{the}$ $(n,q)$-$L^2_\omega$-$\it{estimate}$ $\it{condition}$ on $X$ for $q\geq1$,
    if for any positive holomorphic Hermitian line bundle $(A,h_A)$ on $X$ and for any $f\in\mathcal{D}^{n,q}(X,E\otimes A)$ with $\overline{\partial}f=0$,
    there is $u\in L^2_{n,q-1}(X,E\otimes A)$ satisfying $\overline{\partial}u=f$ and 
    \[ \int_X|u|^2_{h\otimes h_A,\omega}dV_\omega\leq\int_X\langle[i\Theta_{A,h_A}\otimes \mathrm{id}_E,\Lambda_\omega]^{-1}f,f\rangle_{h\otimes h_A,\omega} dV_\omega, \]
    provided that the right hand side is finite.

    And $(E,h)$ satisfies $\it{the}$ $(n,q)$-$L^2$-$\it{estimate~ condition}$ on $X$ if there exists a \kah metric $\tilde{\omega}$ such that $(E,h)$ satisfies the $(n,q)$-$L^2_{\tilde{\omega}}$-estimate condition on $X$
\end{definition}

This definition is an extension of [DNWZ20,\,Definition\,1.1] for $(n,1)$-form to $(n,q)$-forms.
Similarly, we define the following with respect to $(p,n)$-forms.

\begin{definition}\label{def (p,n)-condition}
    Let $(X,\omega)$ be a \kah manifold of dimension $n$ which admits a positive holomorphic Hermitian line bundle and $E$ be a holomorphic vector bundle over $X$ equipped with a (singular) Hermitian metric $h$.
    $(E,h)$ satisfies $\it{the}$ $(p,n)$-$L^2_\omega$-$\it{estimate}$ $\it{condition}$ on $X$ for $p\geq0$,
    if for any positive holomorphic Hermitian line bundle $(A,h_A)$ on $X$ and for any $f\in\mathcal{D}^{p,n}(X,E\otimes A)$ with $\overline{\partial}f=0$,
    there is $u\in L^2_{p,n-1}(X,E\otimes A)$ satisfying $\overline{\partial}u=f$ and 
    \[ \int_X|u|^2_{h\otimes h_A,\omega}dV_\omega\leq\int_X\langle[i\Theta_{A,h_A}\otimes \mathrm{id}_E,\Lambda_\omega]^{-1}f,f\rangle_{h\otimes h_A,\omega} dV_\omega, \]
    provided that the right hand side is finite.

    And $(E,h)$ satisfies $\it{the}$ $(p,n)$-$L^2$-$\it{estimate~ condition}$ on $X$ if there exists \kah metric $\tilde{\omega}$ such that $(E,h)$ satisfies the $(p,n)$-$L^2_{\tilde{\omega}}$-estimate condition on $X$
\end{definition}

Let $(X,\omega)$ be a Hermitian manifold and $(E,h)$ be a holomorphic Hermitian vector bundle over $X$.
We denote the curvature operator $[i\Theta_{E,h},\Lambda_\omega]$ on $\Lambda^{p,q}T^*_X\otimes E$ by $A^{p,q}_{E,h,\omega}$.
And the fact that the curvature operator $[i\Theta_{E,h},\Lambda_\omega]$ is positive (resp. semi-positive) definite on $\Lambda^{p,q}T^*_X\otimes E$ is simply written as $A^{p,q}_{E,h,\omega}>0$ (resp. $\geq 0$).\\

Let $(X,\omega)$ be a \kah manifold of dimension $n$ which admits a positive holomorphic Hermitian line bundle.
In [DNWZ20], by using fact that $(E,h)$ is Nakano semi-positive if and only if $A^{n,1}_{E,h,\omega}\geq0$, they showed the following relationship between the $(n,1)$-$L^2$-estimate condition and Nakano semi-positivity:

$(E,h)$ satisfies the $(n,1)$-$L^2$-estimate condition $\Longrightarrow ~ i\Theta_{E,h}\geq_{Nak}0,$ i.e. $A^{n,1}_{E,h,\omega}\geq0$.\\
We have shown this result for the general case, the $(n,q)$ and $(p,n)$-$L^2$-estimate condition, focusing on $A^{n,q}_{E,h,\omega}\geq0$ and $A^{p,n}_{E,h,\omega}\geq0$ instead of $i\Theta_{E,h}\geq_{Nak}0$.

\begin{theorem}\label{Main thm 1}
    Let $(X,\omega)$ be a \kah manifold of dimension $n$ which admits a positive holomorphic Hermitian line bundle and $(E,h)$ be a holomorphic Hermitian vector bundle over $X$ and $q$ be a positive integer. 
    Then $(E,h)$ satisfies the $(n,q)$-$L^2_\omega$-estimate condition on $X$ if and only if $A^{n,q}_{E,h,\omega}\geq0$.
\end{theorem}

\begin{theorem}\label{Main thm 2}
    Let $(X,\omega)$ be a \kah manifold of dimension $n$ which admits a positive holomorphic Hermitian line bundle and $(E,h)$ be a holomorphic Hermitian vector bundle over $X$ and $p$ be a nonnegative integer. 
    Then $(E,h)$ satisfies the $(p,n)$-$L^2_\omega$-estimate condition on $X$ if and only if $A^{p,n}_{E,h,\omega}\geq0$.
\end{theorem}

And, by studying the properties of the curvature operator, we obtain the following characterizations of Nakano semi-negativity using Theorem \ref{Main thm 2}.

\begin{theorem}\label{Main thm 3}
    Let $(X,\omega)$ be a \kah manifold of dimension $n$ which admits a positive holomorphic Hermitian line bundle and $(E,h)$ be a holomorphic Hermitian vector bundle over $X$. 
    Then $(E^*,h^*)$ satisfies the $(1,n)$-$L^2$-estimate condition on $X$ if and only if $(E,h)$ is Nakano semi-negative.
\end{theorem}

\section{Properties of the curvature operator}\label{section:2}

In this section, we compute the value by the curvature operator in detail and study the properties related to the positivity of the curvature operator.

Let $(M,g)$ be an oriented Riemannian $\mathcal{C}^\infty$-manifold with $\dim_\mathbb{R}M=m$.
Let $\xi\lrcorner~$ be the interior product $\iota_\xi$ for $\xi\in T_M$. If $(\xi_1,\cdots,\xi_m)$ is an orthonormal basis of $(T_M,g)$ at $x_0$, 
then for any ordered multi-index $I=\{i_1,\cdots,i_p\}$ with $i_1<\cdots<i_p$ and $|I|=p\in\{1,\cdots,m\}$, $\bullet \lrcorner~\bullet$ is the bilinear operator characterized by
\begin{align*}
    \xi_s\lrcorner~(\xi_{i_1}^*\wedge\cdots\wedge\xi_{i_p}^*)=
    \begin{cases}
        0 & \mathrm{if}~ s\notin\{i_1,\cdots,i_p\}\\
        (-1)^{k-1}\xi_{i_1}^*\wedge\cdots\widehat{\xi_{i_k}^*}\cdots\wedge\xi_{i_p}^* & \mathrm{if}~s=i_k.
    \end{cases}
\end{align*}

Therefore we introduce a symbol to represent this number $(-1)^{k-1}$ using $s$ and $I$.

\begin{definition}
    Let $(M,g),(\xi_1,\cdots,\xi_m)$ and $I$ be as in above.
    We define $\varepsilon(s,I)\in \{-1,0,1\}$ as the number that satisfies $\xi_s\lrcorner~\xi_I^*=\varepsilon(s,I)\xi^*_{I\setminus s}$, 
    where if $s\notin I$ then $\varepsilon(s,I)=0$ and if $s\in I$ then $\varepsilon(s,I)\in\{-1,1\}$. 

    Let $(X,\omega)$ be a Hermitian manifold, $\dim_\mathbb{C}X=n$. If $(\partial/\partial z_1,\cdots,\partial/\partial z_n)$ is an orthonormal basis of $(T_X,\omega)$ at $x_0$ then we define $\varepsilon(s,I)$ in the same way as follows 
    $\displaystyle \frac{\partial}{\partial z_s}\lrcorner~ dz_I=\varepsilon(s,I)dz_{I\setminus s}$. In particular, we have that $\displaystyle \frac{\partial}{\partial \overline{z}_s}\lrcorner~ d\overline{z}_I=\varepsilon(s,I)d\overline{z}_{I\setminus s}$.
\end{definition}


Using the symbols $\varepsilon(s,I)$ in this definition, we obtain the following detailed calculation results.

\begin{proposition}\label{calculate (p,q)-forms}
    Let $(X,\omega)$ be a Hermitian manifold and $(E,h)$ be a holomorphic Hermitian vector bundle over $X$.
    Let $x_0\in X$ and $(z_1,\ldots,z_n)$ be local coordinates such that $(\partial/\partial z_1,\ldots,\partial/\partial z_n)$ is an orthonormal basis of $(T_X,\omega)$ at $x_0$. Let $(e_1,\ldots,e_r)$ be an orthonormal basis of $E_{x_0}$. We can write
\begin{align*}
    \omega_{x_0}=i\sum_{1\leq j\leq n}dz_j\wedge d\overline{z}_j,\quad
    i\Theta_{E,h,x_0}=i\sum_{j,k,\lambda,\mu}c_{jk\lambda\mu}dz_j\wedge d\overline{z}_k\otimes e^*_\lambda\otimes e_\mu.
\end{align*}
Let $J,K,L$ and $M$ be ordered multi-indices with $|J|=|L|=p$ and $|K|=|M|=q$.
For any $(p,q)$-form $ u=\sum_{|J|=p,|K|=q,\lambda}u_{J,K,\lambda}dz_J\wedge d\overline{z}_K\otimes e_\lambda\in\Lambda^{p,q}T^*_{X,x_0}\otimes E_{x_0}$, we have the following
\begin{align*}
    [i\Theta_{E,h},\Lambda_\omega]u
    &=\Bigl(\sum_{j\in J}+\sum_{j\in K}-\sum_{1\leq j\leq n}\Bigr)c_{jj\lambda\mu}u_{J,K,\lambda}dz_J\wedge d\overline{z}_K\otimes e_\mu\\
    &+\sum_{K\ni j\ne k\notin K}c_{jk\lambda\mu}u_{J,K,\lambda}\varepsilon(j,K)dz_J\wedge d\overline{z}_k\wedge d\overline{z}_{K\setminus j} \otimes e_\mu\\
    &+\sum_{J\ni k\ne j\notin J}c_{jk\lambda\mu}u_{J,K,\lambda}\varepsilon(k,J)dz_j\wedge dz_{J\setminus k}\wedge d\overline{z}_K\otimes e_\mu,\quad and\\
    \langle[i\Theta_{E,h},\Lambda_\omega]u,u\rangle_\omega&=\Bigl(\sum_{j\in J}+\sum_{j\in K}-\sum_{1\leq j\leq n}\Bigr)c_{jj\lambda\mu}u_{J,K,\lambda}\overline{u}_{J,K,\mu}\\
    &+\sum_{j\ne k,K\setminus j=M\setminus k}c_{jk\lambda\mu}u_{J,K,\lambda}\overline{u}_{J,M,\mu}\varepsilon(j,K)\varepsilon(k,M)\\
    &+\sum_{j\ne k,L\setminus j=J\setminus k}c_{jk\lambda\mu}u_{L,K,\lambda}\overline{u}_{J,K,\mu}\varepsilon(k,J)\varepsilon(j,L).
\end{align*}
\end{proposition}

\begin{proof}
    As is well known, we have that
    \begin{align*}
        \Lambda_\omega u=i(-1)^p\sum_{J,K,\lambda,s}u_{J,K,\lambda}\Bigl(\frac{\partial}{\partial z_s}\lrcorner~dz_J\Bigr)\wedge \Bigl(\frac{\partial}{\partial \overline{z}_s}\lrcorner~ d\overline{z}_K\Bigr)\otimes e_\lambda.
    \end{align*}
    Then a simple computation gives
    \begin{align*}
        i\Theta_{E,h}\wedge\Lambda_\omega u&=\sum c_{jk\lambda\mu}u_{J,K,\lambda}dz_j\wedge\Bigl(\frac{\partial}{\partial z_s}\lrcorner~ dz_J\Bigr)\wedge d\overline{z}_k\wedge\Bigl(\frac{\partial}{\partial \overline{z}_s} \lrcorner~ d\overline{z}_K\Bigr) \otimes e_\mu,\\
        \Lambda_\omega\wedge i\Theta_{E,h}u&=\sum c_{jk\lambda\mu}u_{J,K,\lambda}\Bigl(\frac{\partial}{\partial z_s}\lrcorner~ dz_j\wedge dz_J\Bigr)\wedge \Bigl(\frac{\partial}{\partial \overline{z}_s} \lrcorner~ d\overline{z}_k\wedge d\overline{z}_K\Bigr) \otimes e_\mu.
    \end{align*}
    For simplicity, we calculate only the terms in differential form as follows
    \begin{align*}
        \sum dz_j\wedge\Bigl(&\frac{\partial}{\partial z_s}\lrcorner~ dz_J\Bigr)\wedge d\overline{z}_k\wedge\Bigl(\frac{\partial}{\partial \overline{z}_s} \lrcorner~ d\overline{z}_K\Bigr)\\
        =\sum_{s\in J\cap K}\Biggl\{&\sum_{s=j\ne k}dz_J\wedge d\overline{z}_k\wedge\Bigl(\frac{\partial}{\partial \overline{z}_j} \lrcorner~ d\overline{z}_K\Bigr)+\sum_{s=k\ne j}dz_j\wedge\Bigl(\frac{\partial}{\partial z_k}\lrcorner~ dz_J\Bigr)\wedge d\overline{z}_K\\
        +&\sum_{s=j=k}dz_J\wedge d\overline{z}_K+\sum_{s\ne j,k}dz_j\wedge\Bigl(\frac{\partial}{\partial z_s}\lrcorner~ dz_J\Bigr)\wedge d\overline{z}_k\wedge\Bigl(\frac{\partial}{\partial \overline{z}_s} \lrcorner~ d\overline{z}_K\Bigr)\Biggr\},\\
        \sum \Bigl(\frac{\partial}{\partial z_s}\lrcorner~& dz_j\wedge dz_J\Bigr)\wedge \Bigl(\frac{\partial}{\partial \overline{z}_s} \lrcorner~ d\overline{z}_k\wedge d\overline{z}_K\Bigr)\\
        =\sum_{j\notin J,k\notin K}\Biggl\{&\sum_{s=j\ne k}dz_J\wedge \Bigl(\frac{\partial}{\partial \overline{z}_j} \lrcorner~ d\overline{z}_k\wedge d\overline{z}_K\Bigr)+\sum_{s=k\ne j}\Bigl(\frac{\partial}{\partial z_k}\lrcorner~ dz_j\wedge dz_J\Bigr)\wedge d\overline{z}_K\\
        +&\sum_{s=j=k}dz_J\wedge d\overline{z}_K+\sum_{s\ne j,k}\Bigl(\frac{\partial}{\partial z_s}\lrcorner~ dz_j\wedge dz_J\Bigr)\wedge \Bigl(\frac{\partial}{\partial \overline{z}_s} \lrcorner~ d\overline{z}_k\wedge d\overline{z}_K\Bigr)\Biggr\}.
    \end{align*}

    Here if $s\ne j,k$ then we get
    \begin{align*}
        \Bigl(\frac{\partial}{\partial z_s}\lrcorner~ dz_j\wedge dz_J\Bigr)=-dz_j\wedge\Bigl(\frac{\partial}{\partial z_s}\lrcorner~ dz_J\Bigr), ~\Bigl(\frac{\partial}{\partial \overline{z}_s}\lrcorner~ d\overline{z}_k\wedge d\overline{z}_K\Bigr)=-d\overline{z}_k\wedge\Bigl(\frac{\partial}{\partial \overline{z}_s}\lrcorner~ d\overline{z}_K\Bigr). 
    \end{align*}
    From this, we have that
    \begin{align*}
        &\sum_{j\notin J,k\notin K}\sum_{s\ne j,k}\Bigl(\frac{\partial}{\partial z_s}\lrcorner~ dz_j\wedge dz_J\Bigr)\wedge \Bigl(\frac{\partial}{\partial \overline{z}_s} \lrcorner~ d\overline{z}_k\wedge d\overline{z}_K\Bigr)\\
        &=\sum_{j\notin J,k\notin K}\sum_{s\ne j,k}dz_j\wedge\Bigl(\frac{\partial}{\partial z_s}\lrcorner~ dz_J\Bigr)\wedge d\overline{z}_k\wedge\Bigl(\frac{\partial}{\partial \overline{z}_s} \lrcorner~ d\overline{z}_K\Bigr)\\
        &=\sum_{s\in J\cap K}\sum_{s\ne j,k}dz_j\wedge\Bigl(\frac{\partial}{\partial z_s}\lrcorner~ dz_J\Bigr)\wedge d\overline{z}_k\wedge\Bigl(\frac{\partial}{\partial \overline{z}_s} \lrcorner~ d\overline{z}_K\Bigr)
    \end{align*}

    Hence the difference between the two equations is as follows
    \begin{align*}
        \sum\Biggl\{ dz_j\wedge\Bigl(\frac{\partial}{\partial z_s}\lrcorner~ dz_J\Bigr)\wedge d\overline{z}_k\wedge\Bigl(\frac{\partial}{\partial \overline{z}_s} \lrcorner~ d\overline{z}_K\Bigr)-\Bigl(\frac{\partial}{\partial z_s}\lrcorner~ dz_j\wedge dz_J\Bigr)\wedge \Bigl(\frac{\partial}{\partial \overline{z}_s} \lrcorner~ d\overline{z}_k\wedge d\overline{z}_K\Bigr)\Biggr\}
    \end{align*}
    \begin{align*}
        &=\Biggl(\sum_{j=k=s\in J\cap K}+\sum_{j=k=s\notin J\cup K}\Biggr)dz_J\wedge d\overline{z}_K\\
        &\qquad+\sum_{s\in J\cap K}\Biggl\{\sum_{s=j\ne k}dz_J\wedge d\overline{z}_k\wedge\Bigl(\frac{\partial}{\partial \overline{z}_j} \lrcorner~ d\overline{z}_K\Bigr)+\sum_{s=k\ne j}dz_j\wedge\Bigl(\frac{\partial}{\partial z_k}\lrcorner~ dz_J\Bigr)\wedge d\overline{z}_K\Biggr\}\\
        &\qquad-\sum_{j\notin J,k\notin K}\Biggl\{\sum_{s=j\ne k}dz_J\wedge \Bigl(\frac{\partial}{\partial \overline{z}_j} \lrcorner~ d\overline{z}_k\wedge d\overline{z}_K\Bigr)+\sum_{s=k\ne j}\Bigl(\frac{\partial}{\partial z_k}\lrcorner~ dz_j\wedge dz_J\Bigr)\wedge d\overline{z}_K\Biggr\}\\
        &=\Biggl(\sum_{j=k=s\in J\cap K}+\sum_{j=k=s\notin J\cup K}\Biggr)dz_J\wedge d\overline{z}_K\\
        &\qquad+\Biggl(\sum_{s\in J\cap K}\!+\!\sum_{j\notin J,k\notin K}\Biggr)\Biggl\{\sum_{s=j\ne k}dz_J\wedge d\overline{z}_k\wedge\Bigl(\frac{\partial}{\partial \overline{z}_j} \lrcorner~ d\overline{z}_K\Bigr)
        \!+\!\sum_{s=k\ne j}dz_j\wedge\Bigl(\frac{\partial}{\partial z_k}\lrcorner~ dz_J\Bigr)\wedge d\overline{z}_K\Biggr\}\\
        &=\Biggl(\sum_{j=k\in J\cap K}\!\!+\!\!\sum_{j=k\notin J\cup K}\Biggr)dz_J\wedge d\overline{z}_K
        +\Biggl(\sum_{j\in J\cap K}\!+\!\sum_{j\notin J,k\notin K}\Biggr)\!\sum_{j\ne k,k\notin K,j\in K}\!\!dz_J\wedge d\overline{z}_k\wedge\Bigl(\frac{\partial}{\partial \overline{z}_j} \lrcorner~ d\overline{z}_K\Bigr)\\
        &\qquad+\Biggl(\sum_{k\in J\cap K}+\sum_{j\notin J,k\notin K}\Biggr)\sum_{k\ne j,j\notin J,k\in J}dz_j\wedge\Bigl(\frac{\partial}{\partial z_k}\lrcorner~ dz_J\Bigr)\wedge d\overline{z}_K.
    \end{align*}
    We think about the conditions of sigma,
    \begin{align*}
        \Biggl(\sum_{j\in J\cap K}\!+\!\sum_{j\notin J,k\notin K}\Biggr)\!\sum_{j\ne k,k\notin K,j\in K}\!\!&=\!\!\sum_{j\ne k,k\notin K,j\in K}\Biggl(\sum_{j\in J\cap K}+\sum_{j\notin J}\Biggr)
        \!=\!\!\sum_{K\ni j\ne k\notin K}\Biggl(\sum_{j\in J}+\sum_{j\notin J}\Biggr)=\!\!\sum_{K\ni j\ne k\notin K},\\
        \Biggl(\sum_{k\in J\cap K}+\sum_{j\notin J,k\notin K}\Biggr)\sum_{k\ne j,j\notin J,k\in J}&=\sum_{k\ne j,j\notin J,k\in J}\Biggl(\sum_{k\in J\cap K}+\sum_{k\notin K}\Biggr)
        =\!\!\sum_{J\ni k\ne j\notin J}\Biggl(\sum_{k\in K}+\sum_{k\notin K}\Biggr)=\sum_{J\ni k\ne j\notin J}.
    \end{align*}

    Therefore we have that 
    \begin{align*}
        \sum\Biggl\{ dz_j&\wedge\Bigl(\frac{\partial}{\partial z_s}\lrcorner~ dz_J\Bigr)\wedge d\overline{z}_k\wedge\Bigl(\frac{\partial}{\partial \overline{z}_s} \lrcorner~ d\overline{z}_K\Bigr)-\Bigl(\frac{\partial}{\partial z_s}\lrcorner~ dz_j\wedge dz_J\Bigr)\wedge \Bigl(\frac{\partial}{\partial \overline{z}_s} \lrcorner~ d\overline{z}_k\wedge d\overline{z}_K\Bigr)\Biggr\}\\
        &=\Biggl(\sum_{j=k\in J\cap K}+\sum_{j=k\notin J\cup K}\Biggr)dz_J\wedge d\overline{z}_K\\
        &\qquad+\sum_{K\ni j\ne k\notin K}dz_J\wedge d\overline{z}_k\wedge\Bigl(\frac{\partial}{\partial \overline{z}_j} \lrcorner~ d\overline{z}_K\Bigr)+\sum_{J\ni k\ne j\notin J}dz_j\wedge\Bigl(\frac{\partial}{\partial z_k}\lrcorner~ dz_J\Bigr)\wedge d\overline{z}_K\\
        &=\Biggl(\sum_{j=k\in J}+\sum_{j=k\in K}-\sum_{1\leq j=k \leq n}\Biggr)dz_J\wedge d\overline{z}_K\\
        &\qquad+\sum_{K\ni j\ne k\notin K}dz_J\wedge d\overline{z}_k\wedge\Bigl(\frac{\partial}{\partial \overline{z}_j} \lrcorner~ d\overline{z}_K\Bigr)+\sum_{J\ni k\ne j\notin J}dz_j\wedge\Bigl(\frac{\partial}{\partial z_k}\lrcorner~ dz_J\Bigr)\wedge d\overline{z}_K.
    \end{align*}

    From the above, we obtain the first claim as follows:
    \begin{align*}
        [i\Theta_{E,h},\Lambda_\omega]u&=i\Theta_{E,h}\wedge\Lambda_\omega u -\Lambda_\omega\wedge i\Theta_{E,h}u\\
        &=\sum c_{jk\lambda\mu}u_{J,K,\lambda}\Biggl\{ dz_j\wedge\Bigl(\frac{\partial}{\partial z_s}\lrcorner~ dz_J\Bigr)\wedge d\overline{z}_k\wedge\Bigl(\frac{\partial}{\partial \overline{z}_s} \lrcorner~ d\overline{z}_K\Bigr)\\
        &\qquad\qquad\qquad\qquad-\Bigl(\frac{\partial}{\partial z_s}\lrcorner~ dz_j\wedge dz_J\Bigr)\wedge \Bigl(\frac{\partial}{\partial \overline{z}_s} \lrcorner~ d\overline{z}_k\wedge d\overline{z}_K\Bigr)\Biggr\}\otimes e_\mu\\
        &=\Biggl(\sum_{j\in J}+\sum_{j\in K}-\sum_{1\leq j \leq n}\Biggr)c_{jj\lambda\mu}u_{J,K,\lambda}dz_J\wedge d\overline{z}_K\otimes e_\mu\\
        &\qquad+\sum_{K\ni j\ne k\notin K}c_{jk\lambda\mu}u_{J,K,\lambda}dz_J\wedge d\overline{z}_k\wedge\Bigl(\frac{\partial}{\partial \overline{z}_j} \lrcorner~ d\overline{z}_K\Bigr)\otimes e_\mu\\
        &\qquad+\sum_{J\ni k\ne j\notin J}c_{jk\lambda\mu}u_{J,K,\lambda}dz_j\wedge\Bigl(\frac{\partial}{\partial z_k}\lrcorner~ dz_J\Bigr)\wedge d\overline{z}_K\otimes e_\mu\\
        &=\Bigl(\sum_{j\in J}+\sum_{j\in K}-\sum_{1\leq j\leq n}\Bigr)c_{jj\lambda\mu}u_{J,K,\lambda}dz_J\wedge d\overline{z}_K\otimes e_\mu   \tag{\ref{calculate (p,q)-forms}.a} \\
        &\qquad+\sum_{K\ni j\ne k\notin K}c_{jk\lambda\mu}u_{J,K,\lambda}\varepsilon(j,K)dz_J\wedge d\overline{z}_k\wedge d\overline{z}_{K\setminus j} \otimes e_\mu  \tag{\ref{calculate (p,q)-forms}.b}\\
        &\qquad+\sum_{J\ni k\ne j\notin J}c_{jk\lambda\mu}u_{J,K,\lambda}\varepsilon(k,J)dz_j\wedge dz_{J\setminus k}\wedge d\overline{z}_K\otimes e_\mu. \tag{\ref{calculate (p,q)-forms}.c}
    \end{align*}

    We calculate for the equation $\langle[i\Theta_{E,h},\Lambda_\omega]u,u\rangle_\omega$. First,
    \begin{align*}
        \langle(\ref{calculate (p,q)-forms}.\mathrm{a}),u\rangle_\omega&=\langle \Bigl(\sum_{j\in J}+\sum_{j\in K}-\sum_{1\leq j\leq n}\Bigr)c_{jj\lambda\mu}u_{J,K,\lambda}dz_J\wedge d\overline{z}_K\otimes e_\mu\\
        &\qquad\qquad,\sum_{|L|=p,|M|=q,\tau}u_{L,M,\tau} dz_L\wedge d\overline{z}_M\otimes e_\tau \rangle_\omega
    \end{align*}
    \begin{align*}
        &=\langle \Bigl(\sum_{j\in J}+\sum_{j\in K}-\sum_{1\leq j\leq n}\Bigr)c_{jj\lambda\mu}u_{J,K,\lambda}dz_J\wedge d\overline{z}_K,\sum u_{J,K,\mu} dz_J\wedge d\overline{z}_K\rangle_\omega\\
        &=\Bigl(\sum_{j\in J}+\sum_{j\in K}-\sum_{1\leq j\leq n}\Bigr)c_{jj\lambda\mu}u_{J,K,\lambda}\overline{u}_{J,K,\mu},\\
        \langle(\ref{calculate (p,q)-forms}.\mathrm{b}),u\rangle_\omega&=\langle \sum_{K\ni j\ne k\notin K}c_{jk\lambda\mu}u_{J,K,\lambda}\varepsilon(j,K)dz_J\wedge d\overline{z}_k\wedge d\overline{z}_{K\setminus j} \otimes e_\mu,\sum dz_L\wedge d\overline{z}_M\otimes e_\tau \rangle_\omega\\
        &=\langle \sum_{K\ni j\ne k\notin K}c_{jk\lambda\mu}u_{J,K,\lambda}\varepsilon(j,K)dz_J\wedge d\overline{z}_{k,K\setminus j},\sum u_{J,M,\mu} dz_J\wedge d\overline{z}_M\rangle_\omega
    \end{align*}
    Here if $\{k,K\setminus j\}=\{M\}$ as set then we obtain $K\setminus j=M\setminus k$ after also taking into account the order. Therefore we get
    \[
        d\overline{z}_M=d\overline{z}_k\wedge\Bigl(\frac{\partial}{\partial \overline{z}_k} \lrcorner~ d\overline{z}_M\Bigr)=\varepsilon(k,M)d\overline{z}_k\wedge d\overline{z}_{M\setminus k}=\varepsilon(k,M)d\overline{z}_k\wedge d\overline{z}_{K\setminus j}=\varepsilon(k,M)d\overline{z}_{k,K\setminus j}
    \]

    Hence, we have that
    \begin{align*}
        \langle(\ref{calculate (p,q)-forms}.\mathrm{b}),u\rangle_\omega&=\langle \sum_{K\ni j\ne k\notin K}c_{jk\lambda\mu}u_{J,K,\lambda}\varepsilon(j,K)dz_J\wedge d\overline{z}_{k,K\setminus j},\sum_{K\setminus j=M\setminus k} u_{J,M,\mu}\varepsilon(k,M) dz_J\wedge d\overline{z}_{k,K\setminus j}\rangle_\omega\\
        &=\sum_{K\ni j\ne k\notin K,K\setminus j=M\setminus k}c_{jk\lambda\mu}u_{J,K,\lambda}\overline{u}_{J,M,\mu}\varepsilon(j,K)\varepsilon(k,M)\\
        &=\sum_{j\ne k,K\setminus j=M\setminus k}c_{jk\lambda\mu}u_{J,K,\lambda}\overline{u}_{J,M,\mu}\varepsilon(j,K)\varepsilon(k,M),
    \end{align*}
    because if $k\in K$, then $j\ne k$ and $K\setminus j=M\setminus k$ are not satisfied. In a similar way to this equation, we get the following
    \begin{align*}
        \langle(\ref{calculate (p,q)-forms}.\mathrm{c}),u\rangle_\omega&=\sum_{j\ne k,L\setminus j=J\setminus k}c_{jk\lambda\mu}u_{L,K,\lambda}\overline{u}_{J,K,\mu}\varepsilon(k,J)\varepsilon(j,L).
    \end{align*}
    From the above, this proof is completed.
\end{proof}

Let $(X,\omega)$ be a Hermitian manifold and $(E,h)$ be a Hermitian vector bundle over $X$.
We denote the curvature operator $[i\Theta_{E,h},\Lambda_\omega]$ on $\Lambda^{p,q}T^*_X\otimes E$ by $A^{p,q}_{E,h,\omega}$.
For any ordered multi-index $I$, we denote the ordered complementary multi-index of $I$ by $I^C$. Then $\mathrm{sgn}(I,I^C)$ is the signature of the permutation $(1,2,\cdots,n)\to(I,I^C)$.

From Proposition \ref{calculate (p,q)-forms}, we obtain the following theorem, which represents the relation between the positivity and negativity of the curvature operator.

\begin{theorem}\label{(p,q) and (n-q,n-p)}
    Let $(X,\omega)$ be a Hermitian manifold and $(E,h)$ be a Hermitian vector bundle over $X$.
    We have that 
    \[
        A^{p,q}_{E,h,\omega}>0\,\,\, (resp.\, \ge0,\, <0,\, \le0) \iff A^{n-q,n-p}_{E,h,\omega}<0\,\,\, (resp.\, \le0,\, >0,\, \ge0).
    \]
\end{theorem}

\begin{proof}
    Let $x_0\in X$ and $(z_1,\ldots,z_n)$ be local coordinates such that $(\partial/\partial z_1,\ldots,\partial/\partial z_n)$ is an orthonormal basis of $(T_X,\omega)$ at $x_0$. Let $(e_1,\ldots,e_r)$ be an orthonormal basis of $E_{x_0}$. We can write
    \begin{align*}
        \omega_{x_0}=i\sum_{1\leq j\leq n}dz_j\wedge d\overline{z}_j,\quad
        i\Theta_{E,h,x_0}=i\sum_{j,k,\lambda,\mu}c_{jk\lambda\mu}dz_j\wedge d\overline{z}_k\otimes e^*_\lambda\otimes e_\mu.
    \end{align*}

    From Proposition \ref{calculate (p,q)-forms}, for any $(p,q)$-form $ u=\sum_{|J|=p,|K|=q,\lambda}u_{J,K,\lambda}dz_J\wedge d\overline{z}_K\otimes e_\lambda\in\Lambda^{p,q}T^*_{X,x_0}\otimes E_{x_0}$ we have that   
    \begin{align*}
        \langle A^{p,q}_{E,h,\omega}u,u\rangle_\omega&=\Bigl(\sum_{j\in J}+\sum_{j\in K}-\sum_{1\leq j\leq n}\Bigr)c_{jj\lambda\mu}u_{J,K,\lambda}\overline{u}_{J,K,\mu}\\
        &+\sum_{j\ne k,K\setminus j=M\setminus k}c_{jk\lambda\mu}u_{J,K,\lambda}\overline{u}_{J,M,\mu}\varepsilon(j,K)\varepsilon(k,M) \tag{\ref{(p,q) and (n-q,n-p)}.I} \\
        &+\sum_{j\ne k,L\setminus j=J\setminus k}c_{jk\lambda\mu}u_{L,K,\lambda}\overline{u}_{J,K,\mu}\varepsilon(k,J)\varepsilon(j,L).
    \end{align*}

    For the above $u$, we take the $(n-q,n-p)$-form
    \begin{align*}
        \tilde{u}:&=\sum_{|J|=p,|K|=q,\lambda}\mathrm{sgn}(J,J^C)\mathrm{sgn}(K,K^C)(-1)^{|J|}(-1)^{|K|} u_{J,K,\lambda}dz_{K^C}\wedge d\overline{z}_{J^C}\otimes e_\lambda\\
        &=\sum_{|J|=p,|K|=q,\lambda}\alpha(J)\alpha(K) u_{J,K,\lambda}dz_{K^C}\wedge d\overline{z}_{J^C}\otimes e_\lambda,
    \end{align*}
    where $\alpha(J)=\mathrm{sgn}(J,J^C)(-1)^{|J|}$. Then from Proposition \ref{calculate (p,q)-forms}, we have that
    \begin{align*}
        A^{n-q,n-p}_{E,h,\omega}\tilde{u}=\Bigl(\sum_{j\in K^C}+\sum_{j\in J^C}-\sum_{1\leq j\leq n}\Bigr)c_{jj\lambda\mu}\alpha(J)\alpha(K)u_{J,K,\lambda}dz_{K^C}\wedge d\overline{z}_{J^C}\otimes e_\mu \qquad\qquad &(\ref{(p,q) and (n-q,n-p)}.a) \\
        +\sum_{J^C\ni j\ne k\notin J^C}c_{jk\lambda\mu}\alpha(J)\alpha(K)u_{J,K,\lambda}\varepsilon(j,J^C)dz_{K^C}\wedge d\overline{z}_k\wedge d\overline{z}_{J^C\setminus j} \otimes e_\mu \qquad & (\ref{(p,q) and (n-q,n-p)}.b)\\
        +\sum_{K^C\ni k\ne j\notin K^C}c_{jk\lambda\mu}\alpha(J)\alpha(K)u_{J,K,\lambda}\varepsilon(k,K^C)dz_j\wedge dz_{K^C\setminus k}\wedge d\overline{z}_{J^C}\otimes e_\mu \,\quad & (\ref{(p,q) and (n-q,n-p)}.c)
    \end{align*}

    We calculate for the equation $\langle A^{n-q,n-p}_{E,h,\omega}\tilde{u},\tilde{u}\rangle_\omega$. First,
    \begin{align*}
        \langle (\ref{(p,q) and (n-q,n-p)}.\mathrm{a}),\tilde{u}\rangle_\omega&=\langle \Bigl(\sum_{j\in K^C}+\sum_{j\in J^C}-\sum_{1\leq j\leq n}\Bigr)c_{jj\lambda\mu}\alpha(J)\alpha(K)u_{J,K,\lambda}dz_{K^C}\wedge d\overline{z}_{J^C}\otimes e_\mu\\
        &\qquad\qquad,\sum_{|L|=p,|M|=q,\tau}\alpha(L)\alpha(M) u_{L,M,\tau}dz_{M^C}\wedge d\overline{z}_{L^C}\otimes e_\tau \rangle\\
        &=\langle \Bigl(\sum_{j\in K^C}+\sum_{j\in J^C}-\sum_{1\leq j\leq n}\Bigr)c_{jj\lambda\mu}\alpha(J)\alpha(K)u_{J,K,\lambda}dz_{K^C}\wedge d\overline{z}_{J^C}\\
        &\qquad\qquad,\sum \alpha(J)\alpha(K)u_{J,K,\mu}dz_{K^C}\wedge d\overline{z}_{J^C} \rangle\\
        &=\Bigl(\sum_{j\in K^C}+\sum_{j\in J^C}-\sum_{1\leq j\leq n}\Bigr)c_{jj\lambda\mu}\alpha(J)^2\alpha(K)^2u_{J,K,\lambda}\overline{u}_{J,K,\mu}\\
        &=\Bigl(\sum_{1\leq j\leq n}-\sum_{j\in K}+\sum_{1\leq j\leq n}-\sum_{j\in J}-\sum_{1\leq j\leq n}\Bigr)c_{jj\lambda\mu}u_{J,K,\lambda}\overline{u}_{J,K,\mu}\\
        &=-\Bigl(\sum_{j\in J}+\sum_{j\in K}-\sum_{1\leq j\leq n}\Bigr)c_{jj\lambda\mu}u_{J,K,\lambda}\overline{u}_{J,K,\mu}  \tag{\ref{(p,q) and (n-q,n-p)}.II}
    \end{align*}
    \begin{align*}
        \langle (\ref{(p,q) and (n-q,n-p)}.\mathrm{b}),\tilde{u}\rangle_\omega&=\langle \sum_{J^C\ni j\ne k\notin J^C}c_{jk\lambda\mu}\alpha(J)\alpha(K)u_{J,K,\lambda}\varepsilon(j,J^C)dz_{K^C}\wedge d\overline{z}_k\wedge d\overline{z}_{J^C\setminus j} \otimes e_\mu\\
        &\qquad\qquad,\sum_{|L|=p,|M|=q,\tau}\alpha(L)\alpha(M) u_{L,M,\tau}dz_{M^C}\wedge d\overline{z}_{L^C}\otimes e_\tau \rangle\\
        &=\langle \sum_{J^C\ni j\ne k\notin J^C}c_{jk\lambda\mu}\alpha(J)\alpha(K)u_{J,K,\lambda}\varepsilon(j,J^C)dz_{K^C}\wedge d\overline{z}_k\wedge d\overline{z}_{J^C\setminus j}\\
        &\qquad\qquad,\sum\alpha(L)\alpha(K) u_{L,K,\mu}dz_{K^C}\wedge d\overline{z}_{L^C} \rangle
    \end{align*}
    Here, in the same way as in Proposition \ref{calculate (p,q)-forms}, if $\{k,J^C\setminus j\}=\{L^C\}$ as set then we get $J^C\setminus j=L^C\setminus k$ and $d\overline{z}_{L^C}=\varepsilon(k,L^C)d\overline{z}_{k,J^C\setminus j}$. Hence, we have that 
    \begin{align*}
        \langle (\ref{(p,q) and (n-q,n-p)}.\mathrm{b}),\tilde{u}\rangle_\omega&=\langle \sum_{J^C\ni j\ne k\notin J^C}c_{jk\lambda\mu}\alpha(J)\alpha(K)u_{J,K,\lambda}\varepsilon(j,J^C)dz_{K^C}\wedge d\overline{z}_{k,J^C\setminus j}\\
        &\qquad\qquad,\sum_{J^C\setminus j=L^C\setminus k}\alpha(L)\alpha(K) u_{L,K,\mu}\varepsilon(k,L^C)dz_{K^C}\wedge d\overline{z}_{k,J^C\setminus j} \rangle\\
        &=\sum_{J^C\ni j\ne k\notin J^C,J^C\setminus j=L^C\setminus k}c_{jk\lambda\mu}\alpha(J)\alpha(L)\alpha(K)^2u_{J,K,\lambda}\overline{u}_{L,K,\mu}\varepsilon(j,J^C)\varepsilon(k,L^C)\\
        &=\sum_{j\ne k,J^C\setminus j=L^C\setminus k}c_{jk\lambda\mu}\alpha(J)\alpha(L)u_{J,K,\lambda}\overline{u}_{L,K,\mu}\varepsilon(j,J^C)\varepsilon(k,L^C)\\
        &=\sum_{j\ne k,J\setminus k=L\setminus j}c_{jk\lambda\mu}\alpha(J)\alpha(L)u_{J,K,\lambda}\overline{u}_{L,K,\mu}\varepsilon(j,J^C)\varepsilon(k,L^C)
    \end{align*}

    Therefore we prove that if $j\ne k$ and $J\setminus k=L\setminus j$, i.e. $J^C\setminus j=L^C\setminus k$ then 
    \[-\varepsilon(k,J)\varepsilon(j,L)=\alpha(J)\alpha(L)\varepsilon(j,J^C)\varepsilon(k,L^C).\]
    From $j\notin J$, i.e. $j\in J^C$, we have that 
    \begin{align*}
        \frac{\partial}{\partial z_j}\lrcorner~ dz_N&=\mathrm{sgn}(J,J^C)\frac{\partial}{\partial z_j}\lrcorner~ (dz_J\wedge dz_{J^C})\\
        &=\mathrm{sgn}(J,J^C)(-1)^{|J|}dz_J\wedge \Bigl( \frac{\partial}{\partial z_j}\lrcorner~ dz_{J^C}\Bigr)\\
        &=\mathrm{sgn}(J,J^C)(-1)^{|J|}\varepsilon(j,J^C)dz_J\wedge dz_{J^C\setminus j}.
    \end{align*}
    Since $k\in J$, we get
    \begin{align*}
        \frac{\partial}{\partial z_k}\lrcorner~ \Bigl(\frac{\partial}{\partial z_j}\lrcorner~ dz_N\Bigr)&=\mathrm{sgn}(J,J^C)(-1)^{|J|}\varepsilon(j,J^C)\frac{\partial}{\partial z_k}\lrcorner~ \Bigl(dz_J\wedge dz_{J^C\setminus j}\Bigr)\\
        &=\mathrm{sgn}(J,J^C)(-1)^{|J|}\varepsilon(j,J^C)\varepsilon(k,J)dz_{J\setminus k}\wedge dz_{J^C\setminus j}.
    \end{align*}
    In a similar way to this equation, we get the following
    \begin{align*}
        \frac{\partial}{\partial z_j}\lrcorner~ \Bigl(\frac{\partial}{\partial z_k}\lrcorner~ dz_N\Bigr)=\mathrm{sgn}(L,L^C)(-1)^{|L|}\varepsilon(k,L^C)\varepsilon(j,L)dz_{L\setminus j}\wedge dz_{L^C\setminus k}.
    \end{align*}

    From the assumption, we have that 
    \begin{align*}
        \mathrm{sgn}(J,J^C)(-1)^{|J|}&\varepsilon(j,J^C)\varepsilon(k,J)dz_{J\setminus k}\wedge dz_{J^C\setminus j}=\frac{\partial}{\partial z_k}\lrcorner~ \Bigl(\frac{\partial}{\partial z_j}\lrcorner~ dz_N\Bigr)\\
        &=-\frac{\partial}{\partial z_j}\lrcorner~ \Bigl(\frac{\partial}{\partial z_k}\lrcorner~ dz_N\Bigr)\\
        &=-\mathrm{sgn}(L,L^C)(-1)^{|L|}\varepsilon(k,L^C)\varepsilon(j,L)dz_{L\setminus j}\wedge dz_{L^C\setminus k}\\
        &=-\mathrm{sgn}(L,L^C)(-1)^{|L|}\varepsilon(k,L^C)\varepsilon(j,L)dz_{J\setminus k}\wedge dz_{J^C\setminus j}.
    \end{align*}
    Thus, we have shown the following
    \begin{align*}
        \mathrm{sgn}(J,J^C)(-1)^{|J|}\varepsilon(j,J^C)\varepsilon(k,J)&=-\mathrm{sgn}(L,L^C)(-1)^{|L|}\varepsilon(k,L^C)\varepsilon(j,L)\\
        \iff \,\,\,\, -\varepsilon(k,J)\varepsilon(j,L)&=\mathrm{sgn}(J,J^C)(-1)^{|J|}\mathrm{sgn}(L,L^C)(-1)^{|L|}\varepsilon(j,J^C)\varepsilon(k,L^C)\\
        &=\alpha(J)\alpha(L)\varepsilon(j,J^C)\varepsilon(k,L^C).
    \end{align*}

    Hence, we have that 
    \begin{align*}
        \langle (\ref{(p,q) and (n-q,n-p)}.\mathrm{b}),\tilde{u}\rangle_\omega&=\sum_{j\ne k,J\setminus k=L\setminus j}c_{jk\lambda\mu}\alpha(J)\alpha(L)u_{J,K,\lambda}\overline{u}_{L,K,\mu}\varepsilon(j,J^C)\varepsilon(k,L^C)\\
        &=-\sum_{j\ne k,J\setminus k=L\setminus j}c_{jk\lambda\mu}u_{J,K,\lambda}\overline{u}_{L,K,\mu}\varepsilon(k,J)\varepsilon(j,L). \tag{\ref{(p,q) and (n-q,n-p)}.III}
    \end{align*}
    In a similar way to this equation, we get
    \begin{align*}
        \langle (\ref{(p,q) and (n-q,n-p)}.\mathrm{c}),\tilde{u}\rangle_\omega=-\sum_{j\ne k,K\setminus j=M\setminus k}c_{jk\lambda\mu}u_{J,K,\lambda}\overline{u}_{J,M,\mu}\varepsilon(j,K)\varepsilon(k,M). \tag{\ref{(p,q) and (n-q,n-p)}.IV}
    \end{align*}

    By the above conditions (\ref{(p,q) and (n-q,n-p)}.I)-(\ref{(p,q) and (n-q,n-p)}.IV), we have that for any $(p,q)$-form $u\in\Lambda^{p,q}T^*_{X,x_0}\otimes E_{x_0}$, there exists a $(n-q,n-p)$-form $\tilde{u}\in\Lambda^{n-q,n-p}T^*_{X,x_0}\otimes E_{x_0}$ such that
    \[\langle A^{p,q}_{E,h,\omega}u,u\rangle_\omega=-\langle A^{n-q,n-p}_{E,h,\omega}\tilde{u},\tilde{u}\rangle_\omega.\] 
    Therefore we obtain that $A^{n-q,n-p}_{E,h,\omega}>0 \Longrightarrow A^{p,q}_{E,h,\omega}<0$. And the other claim can also be shown in the same way.
\end{proof}

We consider the relationship between the curvature operator and the Hodge-star operator.
Let $(X,\omega)$ be a Hermitian manifold and $E$ be a holomorphic vector bundle over $X$ equipped with a smooth Hermitian metric $h$.
We denote by $L^2_{p,q}(X,E,h,\omega)$ the Hilbert space of $E$-valued $(p,q)$-forms $u$ which satisfy
\[
    ||u||^2_{h,\omega}=\int_X|u|^2_{h,\omega}dV_\omega<+\infty.
\]
Let $\ast_E$ be the Hodge-star operator $L^2_{p,q}(X,E,h,\omega)\rightarrow L^2_{n-p,n-q}(X,E^\ast\!,h^\ast\!,\omega)$ as in \cite{L^2 Serre}, \cite{Dem-book}
and $\overline{\partial}^\ast_E$ be the Hilbert space adjoint to $\overline{\partial}_E:L^2_{p,q}(X,E,h,\omega)\rightarrow L^2_{p,q+1}(X,E,h,\omega)$.
Let $i\Theta_{E,h}$ be the Chern curvature tensor of $(E,h)$ and $\Lambda_\omega$ be the adjoint of multiplication of $\omega$.

\begin{lemma}\label{calculate Hodge *}
    Let $(X,\omega)$ be a Hermitian manifold and $(E,h)$ be a holomorphic Hermitian vector bundle over $X$.
    Let $x_0\in X$ and $(U,(z_1,\ldots,z_n))$ be local coordinates such that $(\partial/\partial z_1,\ldots,\partial/\partial z_n)$ is an orthonormal basis of $(T_X,\omega)$ at $x_0$.
    Then for any $(p,q)$-form $u=\sum_{|J|=p,|K|=q,\lambda}u_{J,K,\lambda}dz_J\wedge d\overline{z}_K\otimes e_\lambda\in\Lambda^{p,q}T^*_{X,x_0}\otimes E_{x_0}$, we have that
    \begin{align*}
        \ast_Eu=\sum_{|J|=p,|K|=q,\lambda}\overline{u}_{J,K,\lambda}C_{J,K}dz_{J^C}\wedge d\overline{z}_{K^C}\otimes e_\lambda^*\in\Lambda^{n-q,n-p}T^*_{X,x_0}\otimes E^*_{x_0},
    \end{align*}
    where $C_{J,K}=i^{n^2}(-1)^{q(n-p)}\mathrm{sgn}(J,J^C)\mathrm{sgn}(K,K^C)$.
\end{lemma}

\begin{proof}
    For this lemma, it suffices to show the following 
    \[
        \ast_E(dz_J\wedge d\overline{z}_K\otimes e_\lambda)=i^{n^2}(-1)^{q(n-p)}\mathrm{sgn}(J,J^C)\mathrm{sgn}(K,K^C)dz_{J^C}\wedge d\overline{z}_{K^C}\otimes e_\lambda^*.
    \]
    Here the Hodge-star operator $\ast_E:\mathcal{E}^{p,q}(X,E,h,\omega)\to\mathcal{E}^{n-p,n-q}(X,E^*,h^*,\omega)$ defined by 
    \[
        \langle \alpha,\beta\rangle_{h,\omega} dV_\omega=\alpha\wedge \ast_E\beta.
    \]

    Therefore we define the constant number $C_{J,K}$ by 
    \[
        \ast_E(dz_J\wedge d\overline{z}_K\otimes e_\lambda)=C_{J,K}dz_{J^C}\wedge d\overline{z}_{K^C}\otimes e_\lambda^*.
    \] 
    From these definitions, we have that 
    \begin{align*}
        i^{n^2}dz_N\wedge d\overline{z}_N&=idz_1\wedge d\overline{z}_1\wedge \cdots \wedge i dz_n\wedge d\overline{z}_n=\frac{\omega^n}{n!}=dV_\omega\\
        &=\langle dz_J\wedge d\overline{z}_K\otimes e_\lambda,dz_J\wedge d\overline{z}_K\otimes e_\lambda \rangle_{h,\omega} dV_\omega\\
        &=dz_J\wedge d\overline{z}_K\otimes e_\lambda \wedge \ast_E(dz_J\wedge d\overline{z}_K\otimes e_\lambda)\\
        &=C_{J,K}dz_J\wedge d\overline{z}_K\wedge dz_{J^C}\wedge d\overline{z}_{K^C}\\
        &=C_{J,K}(-1)^{q(n-p)}dz_J\wedge dz_{J^C}\wedge d\overline{z}_K\wedge d\overline{z}_{K^C}\\
        &=C_{J,K}(-1)^{q(n-p)}\mathrm{sgn}(J,J^C)\mathrm{sgn}(K,K^C)dz_N\wedge d\overline{z}_N.
    \end{align*}

    Hence, we obtain that
    \begin{align*}
        i^{n^2}&=C_{J,K}(-1)^{q(n-p)}\mathrm{sgn}(J,J^C)\mathrm{sgn}(K,K^C),\\
        ~\mathrm{i.e.}~\,C_{J,K}&=i^{n^2}(-1)^{q(n-p)}\mathrm{sgn}(J,J^C)\mathrm{sgn}(K,K^C).
    \end{align*}
    From the above, this proof is completed.
\end{proof}

Using Proposition \ref{calculate (p,q)-forms} and Lemma \ref{calculate Hodge *}, we obtain the following theorem, which shows the relation between the curvature operator and the Hodge-star operator.

\begin{theorem}\label{A_E and A_E^*}
    Let $(X,\omega)$ be a Hermitian manifold and $(E,h)$ be a holomorphic Hermitian vector bundle over $X$.
    Then we have that
    \begin{align*}
        \ast_E[i\Theta_{E,h},\Lambda_{\omega}]= [i\Theta_{E^*,h^*},\Lambda_{\omega}]\ast_E,~ \mathrm{i.e.}~ \ast_EA^{p,q}_{E,h,\omega}=A^{n-p,n-q}_{E^*,h^*,\omega}\ast_E,
    \end{align*}
    and for any $(p,q)$-form $u\in \mathcal{E}^{p,q}(X,E)$ we have that 
    \begin{align*}
        \langle[i\Theta_{E,h},\Lambda_{\omega}]u,u\rangle_{h,\omega}&=\langle[i\Theta_{E^*\!,h^*},\Lambda_{\omega}]\ast_Eu,\ast_Eu\rangle_{h^*\!,\omega}, ~and\\
        \langle[i\Theta_{E,h},\Lambda_{\omega}]^{-1}u,u\rangle_{h,\omega}&=\langle[i\Theta_{E^*\!,h^*},\Lambda_{\omega}]^{-1}\ast_Eu,\ast_Eu\rangle_{h^*\!,\omega}.
    \end{align*}

    Furthermore, it follows that
    \begin{align*}
        A^{p,q}_{E,h,\omega}>0 \,\,\, (resp.\, \ge0,\, <0,\, \le0) \iff & A^{n-p,n-q}_{E^*,h^*,\omega}>0 \,\,\, (resp.\, \ge0,\, <0,\, \le0).
    \end{align*}
\end{theorem}

\begin{proof}
    Let $x_0\in X$ and $(z_1,\ldots,z_n)$ be local coordinates such that $(\partial/\partial z_1,\ldots,\partial/\partial z_n)$ is an orthonormal basis of $(T_X,\omega)$ at $x_0$. 
    Let $(e_1,\ldots,e_r)$ be an orthonormal basis of $E_{x_0}$. We can write $\omega_{x_0}=i\sum_{1\leq j\leq n}dz_j\wedge d\overline{z}_j$ and 
    \begin{align*}
        i\Theta_{E,h,x_0}&=i\sum_{j,k,\lambda,\mu}c_{jk\lambda\mu}dz_j\wedge d\overline{z}_k\otimes e^*_\lambda\otimes e_\mu,\\
        i\Theta_{E^*,h^*,x_0}=-i\Theta_{E,h,x_0}^\dagger&=-i\sum_{j,k,\lambda,\mu}c_{jk\mu\lambda}dz_j\wedge d\overline{z}_k\otimes (e^*_\lambda)^*\otimes e^*_\mu\\
        &=-i\sum_{j,k,\lambda,\mu}c_{jk\lambda\mu}dz_j\wedge d\overline{z}_k\otimes (e^*_\mu)^*\otimes e^*_\lambda.
    \end{align*}

    Let $\displaystyle  u=\sum_{|J|=p,|K|=q,\lambda}u_{J,K,\lambda}dz_J\wedge d\overline{z}_K\otimes e_\lambda\in\Lambda^{p,q}T^*_{X,x_0}\otimes E_{x_0}$.
    By Lemma \ref{calculate Hodge *}, we have
    \begin{align*}
        \ast_Eu=\sum_{|J|=p,|K|=q,\lambda}\overline{u}_{J,K,\lambda}C_{J,K}dz_{J^C}\wedge d\overline{z}_{K^C}\otimes e_\lambda^*\in\Lambda^{n-p,n-q}T^*_{X,x_0}\otimes E^*_{x_0}.
    \end{align*}
    Then since $\overline{c}_{jk\lambda\mu}=c_{kj\mu\lambda}$ and Proposition \ref{calculate (p,q)-forms}, we have that 
    \begin{align*}
        \ast_E[i\Theta_{E,h},\Lambda_{\omega}]u=&\ast_E\Bigl\{\Bigl(\sum_{j\in J}+\sum_{j\in K}-\sum_{1\leq j\leq n}\Bigr)c_{jj\lambda\mu}u_{J,K,\lambda}dz_J\wedge d\overline{z}_K\otimes e_\mu\\
        &+\sum_{K\ni j\ne k\notin K}c_{jk\lambda\mu}u_{J,K,\lambda}\varepsilon(j,K)dz_J\wedge d\overline{z}_{k,K\setminus j} \otimes e_\mu\\
        &+\sum_{J\ni k\ne j\notin J}c_{jk\lambda\mu}u_{J,K,\lambda}\varepsilon(k,J) dz_{j,J\setminus k}\wedge d\overline{z}_K\otimes e_\mu\Bigr\}\\
        =&\Bigl(\sum_{j\in J}+\sum_{j\in K}-\sum_{1\leq j\leq n}\Bigr)\overline{c}_{jj\lambda\mu}\overline{u}_{J,K,\lambda}C_{J,K}dz_{J^C}\wedge d\overline{z}_{K^C}\otimes e^*_\mu\\
        &+\sum_{K\ni j\ne k\notin K}\overline{c}_{jk\lambda\mu}\overline{u}_{J,K,\lambda}C_{J,(k,K\setminus j)}\varepsilon(j,K)dz_{J^C}\wedge d\overline{z}_{(k,K\setminus j)^C} \otimes e^*_\mu\\
        &+\sum_{J\ni k\ne j\notin J}\overline{c}_{jk\lambda\mu}\overline{u}_{J,K,\lambda}C_{(j,J\setminus k),K}\varepsilon(k,J) dz_{(j,J\setminus k)^C}\wedge d\overline{z}_{K^C}\otimes e^*_\mu\\
        =&\Bigl(\sum_{j\in J}+\sum_{j\in K}-\sum_{1\leq j\leq n}\Bigr)c_{jj\mu\lambda}\overline{u}_{J,K,\lambda}C_{J,K}dz_{J^C}\wedge d\overline{z}_{K^C}\otimes e^*_\mu \\
        &+\sum_{K\ni j\ne k\notin K}c_{kj\mu\lambda}\overline{u}_{J,K,\lambda}C_{J,(k,K\setminus j)}\varepsilon(j,K)dz_{J^C}\wedge d\overline{z}_{(k,K\setminus j)^C} \otimes e^*_\mu \\
        &+\sum_{J\ni k\ne j\notin J}c_{kj\mu\lambda}\overline{u}_{J,K,\lambda}C_{(j,J\setminus k),K}\varepsilon(k,J) dz_{(j,J\setminus k)^C}\wedge d\overline{z}_{K^C}\otimes e^*_\mu
    \end{align*}
    \begin{align*}
        =&\Bigl(\sum_{j\in J}+\sum_{j\in K}-\sum_{1\leq j\leq n}\Bigr)c_{jj\mu\lambda}\overline{u}_{J,K,\lambda}C_{J,K}dz_{J^C}\wedge d\overline{z}_{K^C}\otimes e^*_\mu \tag{\ref{A_E and A_E^*}.a}\\
        &+\sum_{K\ni k\ne j\notin K}c_{jk\mu\lambda}\overline{u}_{J,K,\lambda}C_{J,(k,K\setminus j)}\varepsilon(k,K)dz_{J^C}\wedge d\overline{z}_{(j,K\setminus k)^C} \otimes e^*_\mu \tag{\ref{A_E and A_E^*}.b}\\
        &+\sum_{J\ni j\ne k\notin J}c_{jk\mu\lambda}\overline{u}_{J,K,\lambda}C_{(k,J\setminus j),K}\varepsilon(j,J) dz_{(k,J\setminus j)^C}\wedge d\overline{z}_{K^C}\otimes e^*_\mu, ~\, and \tag{\ref{A_E and A_E^*}.c}\\
        [i\Theta_{E^*,h^*},\Lambda_{\omega}]\ast_Eu=&[i\Theta_{E^*,h^*},\Lambda_{\omega}]\sum_{|J|=p,|K|=q,\lambda}\overline{u}_{J,K,\lambda}C_{J,K}dz_{J^C}\wedge d\overline{z}_{K^C}\otimes e_\lambda^*\\
        =&-\Bigl(\sum_{j\in J^C}+\sum_{j\in K^C}-\sum_{1\leq j\leq n}\Bigr)c_{jj\mu\lambda}\overline{u}_{J,K,\lambda}C_{J,K}dz_{J^C}\wedge d\overline{z}_{K^C}\otimes e^*_\mu \tag{\ref{A_E and A_E^*}.a$'$}\\
        &-\sum_{K^C\ni j\ne k\notin K^C}c_{jk\mu\lambda}\overline{u}_{J,K,\lambda}C_{J,K}\varepsilon(j,K^C)dz_{J^C}\wedge d\overline{z}_{k,K^C\setminus j} \otimes e^*_\mu \tag{\ref{A_E and A_E^*}.b$'$}\\
        &-\sum_{J^C\ni k\ne j\notin J^C}c_{jk\mu\lambda}\overline{u}_{J,K,\lambda}C_{J,K}\varepsilon(k,J^C) dz_{j,J^C\setminus k}\wedge d\overline{z}_{K^C}\otimes e^*_\mu \tag{\ref{A_E and A_E^*}.c$'$}
    \end{align*}
    From the condition
    \begin{align*}
        \sum_{j\in J}+\sum_{j\in K}-\sum_{1\leq j\leq n}=\sum_{1\leq j\leq n}-\sum_{j\in J^C}+\sum_{1\leq j\leq n}-\sum_{j\in K^C}-\sum_{1\leq j\leq n}=-\Bigl(\sum_{j\in J^C}+\sum_{j\in K^C}-\sum_{1\leq j\leq n}\Bigr),
    \end{align*}
    we have the equation $(\ref{A_E and A_E^*}.a)=(\ref{A_E and A_E^*}.a')$. We prove that if $j\ne k$ then 
    \[ -\mathrm{sgn}(K,K^C)\varepsilon(j,K^C)d\overline{z}_{k,K^C\setminus j}=\mathrm{sgn}((j,K\setminus k),(j,K\setminus k)^C)\varepsilon(k,K)d\overline{z}_{(j,K\setminus k)^C} \]
    to show that $(\ref{A_E and A_E^*}.b)=(\ref{A_E and A_E^*}.b')$ and $(\ref{A_E and A_E^*}.c)=(\ref{A_E and A_E^*}.c')$. 
    From definition, we have that $\varepsilon(s,K)d\overline{z}_s\wedge d\overline{z}_{K\setminus s}=d\overline{z}_K$ and that $\mathrm{sgn}(K,K^C)d\overline{z}_{K^C}$ is characterized by $\mathrm{sgn}(K,K^C)d\overline{z}_{K}\wedge d\overline{z}_{K^C}=d\overline{z}_N$. 
    Then the above equation are proven from the following equation:
    \begin{align*}
        -\mathrm{sgn}(K,K^C)&\varepsilon(j,K^C)\varepsilon(k,K)d\overline{z}_{j,K\setminus k}\wedge d\overline{z}_{k,K^C\setminus j}\\
        &=-\mathrm{sgn}(K,K^C)\varepsilon(j,K^C)\varepsilon(k,K) d\overline{z}_j\wedge d\overline{z}_{K\setminus k}\wedge  d\overline{z}_k\wedge d\overline{z}_{K^C\setminus j}\\
        &=\mathrm{sgn}(K,K^C)\varepsilon(j,K^C)\varepsilon(k,K) d\overline{z}_k\wedge d\overline{z}_{K\setminus k}\wedge  d\overline{z}_j\wedge d\overline{z}_{K^C\setminus j}\\
        &=d\overline{z}_N.
    \end{align*}
    Therefore we get the following equation
    \begin{align*}
        C_{J,(j,K\setminus k)}&\varepsilon(k,K)dz_{J^C}\wedge d\overline{z}_{(j,K\setminus k)^C}\\
        &=i^{n^2}(-1)^{q(n-p)}\mathrm{sgn}(J,J^C)\mathrm{sgn}((k,j\setminus k),(k,j\setminus k)^C)\varepsilon(k,K)dz_{J^C}\wedge d\overline{z}_{(j,K\setminus k)^C}\\
        &=-i^{n^2}(-1)^{q(n-p)}\mathrm{sgn}(J,J^C)\mathrm{sgn}(K,K^C)\varepsilon(j,K^C)dz_{J^C}\wedge d\overline{z}_{k,K^C\setminus j}\\
        &=-C_{J,K}\varepsilon(j,K^C)dz_{J^C}\wedge d\overline{z}_{k,K^C\setminus j}.
    \end{align*}

    Thus from this equation and the fact that $K\ni k\ne j\notin K \iff K^C\ni j\ne k\notin K^C$, we have that $(\ref{A_E and A_E^*}.b)=(\ref{A_E and A_E^*}.b')$.
    And the other equation $(\ref{A_E and A_E^*}.c)=(\ref{A_E and A_E^*}.c')$ can also be shown in the same way. Hence, from the above we have that
    \begin{align*}
        \ast_E\,[i\Theta_{E,h},\Lambda_{\omega}]= [i\Theta_{E^*,h^*},\Lambda_{\omega}]\ast_E.
    \end{align*}

    For any $(p,q)$-form $u\in \mathcal{E}^{p,q}(X,E)$, we have that $\langle[i\Theta_{E,h},\Lambda_{\omega}]u,u\rangle_{h,\omega}=\langle[i\Theta_{E^*\!,h^*},\Lambda_{\omega}]\ast_Eu,\ast_Eu\rangle_{h^*\!,\omega}$ from the following calculation:
    \begin{align*}
        \langle[i\Theta_{E,h},\Lambda_{\omega}]u,u\rangle_{h,\omega}dV_\omega&=\overline{\langle u,[i\Theta_{E,h},\Lambda_{\omega}]u\rangle}_{h,\omega}dV_\omega=\overline{u\wedge \ast_E[i\Theta_{E,h},\Lambda_{\omega}]u}\\
        &=\overline{u\wedge[i\Theta_{E^*,h^*},\Lambda_{\omega}]\ast_Eu}=(-1)^{(n+p)(n-p)}[i\Theta_{E^*,h^*},\Lambda_{\omega}]\ast_Eu\wedge u\\
        &=(-1)^{(p+q)(n-p+n-q)}[i\Theta_{E^*,h^*},\Lambda_{\omega}]\ast_Eu\wedge\ast_{E^*}\ast^{-1}_{E^*}\ast^{-1}_E\ast_Eu\\
        &=(-1)^{(p+q)^2}(-1)^{n-p+n-q}[i\Theta_{E^*,h^*},\Lambda_{\omega}]\ast_Eu\wedge\ast_{E^*}\ast_Eu\\
        &=\langle[i\Theta_{E^*,h^*},\Lambda_{\omega}]\ast_Eu,\ast_Eu\rangle_{h^*,\omega}dV_\omega.
    \end{align*}

    And, we obtain that $\langle[i\Theta_{E,h},\Lambda_{\omega}]^{-1}u,u\rangle_{h,\omega}=\langle[i\Theta_{E^*\!,h^*},\Lambda_{\omega}]^{-1}\ast_Eu,\ast_Eu\rangle_{h^*\!,\omega}$ as follows:\\
    Let $v:=[i\Theta_{E,h},\Lambda_{\omega}]^{-1}u$ then $u=[i\Theta_{E,h},\Lambda_{\omega}]v$. Therefore we get
    \begin{align*}
        [i\Theta_{E^*\!,h^*},\Lambda_{\omega}]\ast_Ev&=\ast_E[i\Theta_{E,h},\Lambda_{\omega}]v=\ast_Eu,\\
        [i\Theta_{E^*\!,h^*},\Lambda_{\omega}]^{-1}\ast_Eu&=[i\Theta_{E^*\!,h^*},\Lambda_{\omega}]^{-1}\ast_E[i\Theta_{E,h},\Lambda_{\omega}]v\\
        &=[i\Theta_{E^*\!,h^*},\Lambda_{\omega}]^{-1}[i\Theta_{E^*\!,h^*},\Lambda_{\omega}]\ast_Ev=\ast_Ev.
    \end{align*}

    Hence, we have that
    \begin{align*}
        \langle[i\Theta_{E,h},\Lambda_{\omega}]^{-1}u,u\rangle_{h,\omega}&=\langle v,[i\Theta_{E,h},\Lambda_{\omega}]v\rangle_{h,\omega}=\overline{\langle [i\Theta_{E,h},\Lambda_{\omega}]v,v\rangle}_{h,\omega}\\
        &=\overline{\langle [i\Theta_{E^*\!,h^*},\Lambda_{\omega}]\ast_Ev,\ast_Ev\rangle}_{h^*,\omega}=\langle \ast_Ev,[i\Theta_{E^*\!,h^*},\Lambda_{\omega}]\ast_Ev\rangle_{h^*,\omega}\\
        &=\langle[i\Theta_{E^*\!,h^*},\Lambda_{\omega}]^{-1}\ast_Eu,\ast_Eu\rangle_{h^*\!,\omega}.
    \end{align*}
    From the above, this proof is completed.
\end{proof}

For the $L^2$-estimate that is of $\overline{\partial}$-type, we obtain the $\overline{\partial}^*$-type $L^2$-estimate in the same way from equation $\overline{\partial}^*\circ\overline{\partial}^*=0$. 
By using Theorem \ref{A_E and A_E^*}, we obtain a more direct relationship, as follows.

\begin{corollary}
    Let $(X,\omega)$ be a Hermitian manifold and let $(E,h)$ be a holomorphic Hermitian vector bundle over $X$.
    If $\omega$ is complete on $X$ and the curvature operator $[i\Theta_{E,h},\Lambda_\omega]$ is positive definite on $\Lambda^{p,q}T^*_X\otimes E$, i.e. $A^{p,q}_{E,h,\omega}>0$, for some $q\geq1$, 
    then for any form $g\in L^2_{p,q}(X,E,h,\omega)$ satisfying $\overline{\partial}^*_Eg=0$ and $\int_X\langle [i\Theta_{E,h},\Lambda_\omega]^{-1}g,g\rangle_{h,\omega}dV_\omega<+\infty$, there exists $f\in L^2_{p,q+1}(X,E,h,\omega)$ such that $\overline{\partial}^*_Ef=g$ and
    \[
        ||f||^2_{h,\omega}\leq\int_X\langle [i\Theta_{E,h},\Lambda_\omega]^{-1}g,g\rangle_{h,\omega}dV_\omega.
    \]

    In particular, we have that $\ast_Eg\in L^2_{n-p,n-q}(X,E^*,h^*,\omega)$ satisfying $\overline{\partial}_{E^*}\ast_Eg=0$ and
    \[
        \int_X\langle [i\Theta_{E^*,h^*},\Lambda_\omega]^{-1}\ast_Eg,\ast_Eg\rangle_{h^*,\omega}dV_\omega=\int_X\langle [i\Theta_{E,h},\Lambda_\omega]^{-1}g,g\rangle_{h,\omega}dV_\omega<+\infty,
    \]
    and that $\ast_Ef\in L^2_{n-p,n-q-1}(X,E^*,h^*,\omega)$ satisfying $(-1)^{p+q+1}\overline{\partial}_{E^*}\ast_Ef=\ast_Eg$ and
    \begin{align*}
        ||\ast_Ef||^2_{h^*,\omega}=||f||^2_{h,\omega}&\leq\int_X\langle [i\Theta_{E,h},\Lambda_\omega]^{-1}g,g\rangle_{h,\omega}dV_\omega\\
        &=\int_X\langle [i\Theta_{E^*,h^*},\Lambda_\omega]^{-1}\ast_Eg,\ast_Eg\rangle_{h^*,\omega}dV_\omega.
    \end{align*}
        
    In other words, when the solution to the $\overline{\partial}^*$-type $L^2$-estimate for $(p,q)$-form $g$ is $f$, the solution to the $L^2$-estimate for $\ast_Eg$ can be given by $(-1)^{p+q+1}\ast_Ef$.
\end{corollary}

\begin{proof}
    From the formula $\overline{\partial}^*_E=-\ast_{E^*}\overline{\partial}_{E^*}\ast_E$ (cf. \cite{L^2 Serre}), we get $\overline{\partial}_{E^*}\ast_Eg=0$. Since $\ast_{E^*}\ast_E|_{\mathcal{E}^{p,q}(X,E)}=(-1)^{p+q}$ (cf. \cite{L^2 Serre}), we have that
    \[
        \ast_Eg=\ast_E\overline{\partial}^*_Ef=-\ast_E\ast_{E^*}\overline{\partial}_{E^*}\ast_Ef=-(-1)^{n-p+n-q}\overline{\partial}_{E^*}\ast_Ef=(-1)^{p+q+1}\overline{\partial}_{E^*}\ast_Ef.
    \]
    From the above and Theorem \ref{A_E and A_E^*}, this proof is completed.
\end{proof}

\section{The $(n,q)$ and $(p,n)$-$L^2$-estimate condition and semi-positivity of the curvature operator}\label{section:3}

In this section, we prove Theorem \ref{Main thm 1} and \ref{Main thm 2} which is an extension of \,[DNWZ20,\\
Theorem\,1.1] 
by using the following proposition.

\begin{proposition}\label{posi line bdl metric}$(\mathrm{cf.~[DNWZ20,~Proposition ~2.1]})$
    Let $X$ be a \kah manifold, which admits a positive holomorphic Hermitian line bundle, and $(A,h_A)$ be a positive holomorphic Hermitian line bundle over $X$.
    Let $(U,(z_1,\cdots,z_n))$ ba a local coordinate on $X$, such that $A|_U$ is trivial, and $B\subset\subset U$ be a coordinate ball.
    Then for any smooth strictly plurisubharmonic function $\psi$ on $U$, there is a positive integer $m$, and a Hermitian metric $h_m$ on the line bundle $A^{\otimes m}$ such that $h_m=e^{-\psi_m}$ on $U$ with $\psi_m|_B=\psi$, where $\psi_m$ is a smooth strictly plurisubharmonic function.
\end{proposition}

\begin{proof}
    Assume that $h_A|_U=e^{-\phi}$ for some smooth strictly plurisubharmonic function $\phi$ on $U$. We may assume that $\phi>0$. We may assume that $B:=B_1$ is the unit ball, and the ball $B_{1+3\delta}$ with radius $1+3\delta$ is also contained in $U$, for $0<\delta<<1$.
    Let $\chi$ be a cut-off function on $U$ such that $\chi|_{\overline{B}_{1+\delta}}=1$ and $\chi|_{U\setminus B_{1+2\delta}}=0$. Let $\phi_m:=m\phi+\chi\log(||z||^2-1)$ on $U\setminus B_1$, where $m>>1$ is an integer such that $\psi_m$ is strictly plurisubharmonic on $U\setminus B_1$ and $\phi_m>\psi$ on $\partial B_{1+\delta}$.

    Now we define a function $\psi_m$ on $U$ as follows:
    \begin{equation*}
        \psi_m(z) = \left\{  
        \begin{aligned}
        &\phi_m, & \text{outside $B_{1+\delta}$}; \\
        &\max_\varepsilon\{\phi_m, \psi\}, & \text{on $B_{1+\delta}\setminus B_1$}; \\
        &\psi, & \text{on $B_1$}.
        \end{aligned}
        \right.
    \end{equation*}

    Then for $0<\varepsilon<<1, \psi_m$ is strictly plurisubharmonic on $U$, $\psi_m|_B=\psi$, and equals to $m\psi$ on $U\setminus B_{1+2\delta}$. So $\psi_m$ gives a Hermitian metric on $A^{\otimes m}|_U$ which coincides with $h^{\otimes m}$ on $U\setminus B_{1+2\delta}$.
\end{proof}

\begin{theorem}\label{L^2-estimate completeness (p,n)-forms}$(\mathrm{cf.~[Wat22,~Theorem ~3.7]})$
    Let $(X,\widehat{\omega})$ be a complete \kah manifold, $\omega$ be another \kah metric which is not necessarily complete and $(E,h)$ be a holomorphic vector bundle which satisfies $A^{p,n}_{E,h,\omega}\geq0$. 
    Then for any $\overline{\partial}$-closed $f\in L^2_{p,n}(X,E,h,\omega)$ there exists $u\in L^2_{p,n-1}(X,E,h,\omega)$ satisfies $\overline{\partial}u=f$ and 
    \begin{align*}
        \int_X|u|^2_{h,\omega}dV_{\omega}\leq\int_X\langle(A^{p,n}_{E,h,\omega})^{-1}f,f\rangle_{h,\omega}dV_{\omega},
    \end{align*}
    where we assume that the right-hand side is finite.
\end{theorem}

Using Proposition \ref{posi line bdl metric} and Theorem \ref{L^2-estimate completeness (p,n)-forms},
we obtain the following characterization of $A^{p,n}_{E,h,\omega}\geq0$ by $L^2$-estimate.
Moreover, by this theorem, we obtain the characterization of Nakano semi-negativity by $L^2$-estimate in the next section.

\begin{theorem}$(\mathrm{= Theorem ~\ref{Main thm 2}})$
    Let $(X,\omega)$ be a \kah manifold of dimension $n$ which admits a positive holomorphic Hermitian line bundle and $(E,h)$ be a holomorphic Hermitian vector bundle over $X$ and $p$ be a nonnegative integer. 
    Then $(E,h)$ satisfies the $(p,n)$-$L^2_\omega$-estimate condition on $X$ if and only if $A^{p,n}_{E,h,\omega}\geq0$.
\end{theorem}

\begin{proof}
    First, we show $A^{p,n}_{E,h,\omega}\geq0$ if $(E,h)$ satisfies the $(p,n)$-$L^2_\omega$-estimate condition on $X$.
    Let $(A,h_A)$ be a positive holomorphic Hermitian line bundle over $X$. Let $(U,(z_1,\cdots,z_n))$ ba a local coordinate on $X$, such that $A|_U$ is trivial, and $B\subset\subset U$ be a coordinate ball.
    From Proposition \ref{posi line bdl metric}, for any smooth strictly plurisubharmonic function $\psi$ on $U$, there is a positive integer $l$, and a Hermitian metric $h_l$ on the line bundle $A^{\otimes l}$, such that $h_l=e^{-\tilde{\psi}}$ on $U$ with $\tilde{\psi}|_B=\psi$, where $\tilde{\psi}$ is a smooth strictly plurisubharmonic function.
    By assumption, for any $f\in\mathcal{D}^{p,q}(X,E\otimes A^{\otimes l})$ such that $\overline{\partial}f=0$ and $\mathrm{supp}\, f\subset U$, there is $u\in L^2_{p,q-1}(X,E\otimes A^{\otimes l})$ such that $\overline{\partial}u=f$ and
    \[ 
        ||u||^2_{h\otimes h_l,\omega}=\int_X|u|^2_{h\otimes h_l,\omega}dV_\omega\leq\int_X\langle B^{-1}_{\tilde{\psi}}f,f\rangle_{h,\omega}e^{\tilde{\psi}}dV_\omega, 
    \]
    where $B_{\tilde{\psi}}:=[\idd\tilde{\psi}\otimes\mathrm{id}_E,\Lambda_\omega]$. Here, from $\mathrm{supp}\, f\subset U$ and $A|_U$ is trivial, we have that $f\in\mathcal{D}^{p,q}(U,E\otimes A^{\otimes l})\subset\mathcal{D}^{p,q}(X,E)$ and 
    \[
        \int_X\langle B^{-1}_{\tilde{\psi}}f,f\rangle_{h,\omega}e^{\tilde{\psi}}dV_\omega=\int_X\langle [i\Theta_{A,h_l}\otimes\mathrm{id}_E,\Lambda_\omega]^{-1}f,f\rangle_{h\otimes h_l,\omega}dV_\omega.
    \]

    From above estimate and Bochner-Kodaira-Nakano identity, for any $\alpha\in\mathcal{D}^{p,q}(X,E)$ such that $\mathrm{supp}\, \alpha\subset U$, we have that
    \begin{align*}
        &\left|\int_X\langle f,\alpha\rangle_{h,\omega}e^{-\tilde{\psi}}dV_\omega\right|^2:=|\langle\langle f,\alpha\rangle\rangle_{\tilde{\psi}}|^2=|\langle\langle\overline{\partial}u,\alpha\rangle\rangle_{\tilde{\psi}}|^2=|\langle\langle u,\overline{\partial}^*\alpha\rangle\rangle_{\tilde{\psi}}|^2\leq ||u||^2_{h\otimes h_l}||\overline{\partial}^*\alpha||^2_{\tilde{\psi}}\\
        &\leq \int_X\langle B^{-1}_{\tilde{\psi}}f,f\rangle_{h,\omega}e^{\tilde{\psi}}dV_\omega\\
        &\qquad \times\bigl(||D'\alpha||^2_{\tilde{\psi}}+||D'^*\alpha||^2_{\tilde{\psi}}-||\overline{\partial}\alpha||^2_{\tilde{\psi}}+\langle\langle[i\Theta_{E,h}+\idd\tilde{\psi}\otimes\mathrm{id}_E,\Lambda_\omega]\alpha,\alpha\rangle\rangle_{\tilde{\psi}}\bigr)\\
        &\leq\int_X\langle B^{-1}_{\tilde{\psi}}f,f\rangle_{h,\omega}e^{\tilde{\psi}}dV_\omega\cdot\bigl(\langle\langle[i\Theta_{E,h}+\idd\tilde{\psi}\otimes\mathrm{id}_E,\Lambda_\omega]\alpha,\alpha\rangle\rangle_{\tilde{\psi}}+||D'\alpha||^2_{\tilde{\psi}}+||D'^*\alpha||^2_{\tilde{\psi}}\bigr),
    \end{align*}
    where $D'$ is the $(1,0)$ part of the Chern connection on $E\otimes A^{\otimes l}$ with respect to the metric $h\otimes h_l$. In particular, $D'|_U$ is also the $(1,0)$ part of the Chern connection on $E\otimes A^{\otimes l}|_U=E|_U$ with respect to the metric $h\otimes h_l|_U=he^{-\tilde{\psi}}$.

    Let $\alpha=B^{-1}_{\tilde{\psi}}f$, i.e. $f=B_{\tilde{\psi}}\alpha$. Then the above inequality becomes
    \begin{align*}
        &\bigl(\langle\langle B_{\tilde{\psi}}\alpha,\alpha\rangle\rangle_{\tilde{\psi}}\bigr)^2\\
        &\leq\langle\langle \alpha,B_{\tilde{\psi}}\alpha\rangle\rangle_{\tilde{\psi}}\bigl(\langle\langle[i\Theta_{E,h}+\idd\tilde{\psi}\otimes\mathrm{id}_E,\Lambda_\omega]\alpha,\alpha\rangle\rangle_{\tilde{\psi}}+||D'\alpha||^2_{\tilde{\psi}}+||D'^*\alpha||^2_{\tilde{\psi}}\bigr)\\
        &=\langle\langle \alpha,B_{\tilde{\psi}}\alpha\rangle\rangle_{\tilde{\psi}}\bigl(\langle\langle[i\Theta_{E,h},\Lambda_\omega]\alpha,\alpha\rangle\rangle_{\tilde{\psi}}+\langle\langle B_{\tilde{\psi}}\alpha,\alpha\rangle\rangle_{\tilde{\psi}}+||D'\alpha||^2_{\tilde{\psi}}+||D'^*\alpha||^2_{\tilde{\psi}}\bigr).
    \end{align*}
    Therefore we get
    \begin{align*}
        \langle\langle[i\Theta_{E,h},\Lambda_\omega]\alpha,\alpha\rangle\rangle_{\tilde{\psi}}+||D'\alpha||^2_{\tilde{\psi}}+||D'^*\alpha||^2_{\tilde{\psi}}\geq0. \tag{$\ast$}
    \end{align*}
    Using this formula $(\ast)$, we show the theorem by contradiction.

    Suppose that $A^{p,n}_{E,h,\omega}$ is not semi-positive on $X$. Then there is $x_0\in X$ and $\xi_0\in \Lambda^{p,n}T_{X,x_0}\otimes E_{x_0}$ such that $|\xi_0|=1$ and $\langle[i\Theta_{E,h},\Lambda_\omega]\xi_0,\xi_0\rangle_{h,\omega}=-2c$ for some $c>0$.

    For any small number $\varepsilon>0$, let $M_\varepsilon$ is a regularized max function (see [Dem-book,\,ChapterI,\,Section5]).
    Here, the function $M_\varepsilon$ possesses the following properties (see the proof of [Wat21,\,Proposition\,3.2]):
    \begin{itemize}
        \item [(a)] $M_\varepsilon(x,y)$ is non-decreasing in all variables, smooth and convex on $\mathbb{R}^2$,
        \item [(b)] $\max\{x,y\}\leq M_\varepsilon(x,y)\leq\max\{x,y\}+\varepsilon$ and
        \item [(c)] $M_\varepsilon(x,y)=\max\{x,y\}$ on $\{(x,y)\in\mathbb{R}^2\mid|x-y|\geq2\varepsilon\}$.
    \end{itemize}

    Let $(U,(z_1,\cdots,z_n))$ be a holomorphic coordinate in $X$ centered at $x_0$ such that $\omega=i\sum dz_j\wedge d\overline{z}_j+O(|z|^2)$. 
    For any small number $R>0$, we define $B_R:=\{z\in U\mid |z|<R\}$ such that $B_{2R}\subset U$. Let $\varphi=|z|^2-R^2$ and $\varphi_\varepsilon=|z|^2-(R^2+\varepsilon^2)$, where $\varepsilon>0$ is enough small with respect to $R$.
    Then we define the smooth strictly plurisubharmonic function $\psi_m$ on $B_{2R}$ by $\psi_m:=M_{\varepsilon/2}(\varphi,m\varphi_\varepsilon)$. Since the above conditions $(b)$ and $(c)$, for any $m\geq 4$ we have that $\psi_m=\max\{\varphi,m\varphi_\varepsilon\}$ on $(B_{R+2\varepsilon}\setminus B_R)^c$, i.e.
    \begin{align*}
        \psi_m|_{B_R}=\varphi<0,\qquad \psi_m|_{B_{2R}\setminus B_{R+2\varepsilon}}=m\varphi_\varepsilon>0, 
    \end{align*}
    and that $\max\{\varphi,m\varphi_\varepsilon\}\leq\psi_m\leq\max\{\varphi,m\varphi_\varepsilon\}+\varepsilon$.

    There exists a holomorphic frame $(e_1,\cdots,e_r)$ of $E$ on $U$ such that $h=I+O(|z|^2)$ at $z=0$ (see [Wel80,\,ChapterIII,\,Lemma\,2.3]). 
    Then we have that 
    \begin{align*}
        D'_h&=\partial+h^{-1}\partial h=\partial+(I+O(|z|^2))O(|z|)=\partial+O(|z|),~ \mathrm{and}\\
        D'_{e^{-m\varphi}}&=\partial+e^{m\varphi}\partial e^{-m\varphi}=\partial-\partial m\varphi=\partial-m\sum\overline{z}_jdz_j.
    \end{align*}
    From Proposition \ref{posi line bdl metric}, we can construct a Hermitian metric $h_l$ on the line bundle $A^{\otimes l}$ such that $h_l=e^{-\tilde{\psi}_m}$ on $U$ with $\tilde{\psi}_m|_{B_{2R}}=\psi_m$.
    Hence, we get 
    \begin{align*}
        D'|_{B_R}&=D'_{he^{-\varphi}}=\partial-\sum\overline{z}_jdz_j+O(|z|),\\ 
        D'|_{B_{2R}\setminus B_{R+2\varepsilon}}&=D'_{he^{-m\varphi_\varepsilon}}=\partial-m\sum\overline{z}_jdz_j+O(|z|),
    \end{align*}
    where $O(|z|)$ is independent of $m$. From the fact
    \begin{align*}
        \omega&=i\sum dz_j\wedge d\overline{z}_j+O(|z|^2)=\idd\varphi+O(|z|^2),\\
        B_{\varphi}&=[\idd\varphi\otimes\mathrm{id}_E,\Lambda_\omega]=[\omega\otimes\mathrm{id}_E+O(|z|^2),\Lambda_\omega]\\
        &=p\cdot\mathrm{id}_E+O(|z|^2)\,\,\,\,\,\,\,\,~\mathrm{on}~\,\Lambda^{p,n}T_X\otimes E,
    \end{align*}
    we have that 
    \begin{align*}
        B_{\widetilde{\psi}_m}|_{B_R}&=B_{\varphi}=p\cdot\mathrm{id}_E+O(|z|^2),\\
        B_{\widetilde{\psi}_m}|_{B_{2R}\setminus B_{R+2\varepsilon}}&=B_{m\varphi_\varepsilon}=[m\idd\varphi_\varepsilon\otimes\mathrm{id}_E,\Lambda_\omega]=mp\cdot\mathrm{id}_E+O(|z|^2)
    \end{align*}
    on $\Lambda^{p,n}T_X\otimes E$.

    Let $\xi=\sum \xi_{J,K,\lambda}dz_J\wedge d\overline{z}_K\otimes e_\lambda\in\mathcal{E}^{p,n}(U,E)$, with constant coefficients such that $\xi(x_0)=\xi_0$. We may assume
    \begin{align*}
        \langle[i\Theta_{E,h},\Lambda_\omega]\xi,\xi\rangle_{h,\omega}<-c
    \end{align*}
    on $U$. For any small number $R>0$.

    Choose $\chi\in\mathcal{D}(B_{2R},\mathbb{R}_{\geq 0})$ such that $\chi|_{B_{R+2\varepsilon}}=1$. Let 
    \[ 
        v=\frac{1}{n}\sum_{J,N,\lambda,j} (-1)^{|J|}\varepsilon(j,N)\overline{z}_j\xi_{J,N,\lambda}\chi(z)dz_J\wedge d\overline{z}_{N\setminus j}\otimes e_\lambda\in\mathcal{D}^{p,n-1}(X,E),
    \]
    then from $(-1)^{|J|}\varepsilon(j,N)d\overline{z}_j\wedge dz_J\wedge d\overline{z}_{N\setminus j}=dz_J\wedge d\overline{z}_N$, we have that 
    \begin{align*}
        \overline{\partial}v|_{B_{R+2\varepsilon}}&=\frac{1}{n}\overline{\partial}\sum_{J,N,\lambda,j} (-1)^{|J|}\varepsilon(j,N)\overline{z}_j\xi_{J,N,\lambda}dz_J\wedge d\overline{z}_{N\setminus j}\otimes e_\lambda\\
        &=\frac{1}{n}\sum_{J,N,\lambda,j} (-1)^{|J|}\varepsilon(j,N)\xi_{J,N,\lambda}\sum^n_{l=1}\frac{\partial}{\partial\overline{z}_l}\overline{z}_jdz_l\wedge dz_J\wedge d\overline{z}_{N\setminus j}\otimes e_\lambda\\
        &=\frac{1}{n}\sum_{J,N,\lambda,j} (-1)^{|J|}\varepsilon(j,N)\xi_{J,N,\lambda}dz_j\wedge dz_J\wedge d\overline{z}_{N\setminus j}\otimes e_\lambda\\
        &=\frac{1}{n}\sum_{J,N,\lambda}\sum_{j\in N} \xi_{J,N,\lambda}dz_J\wedge d\overline{z}_N\otimes e_\lambda\\
        &=\xi,
    \end{align*} 
    where if $j\notin N$ then $\varepsilon(j,N)=0$.
    Therefore let $f:=\overline{\partial}v\in\mathcal{D}^{p,n}(U,E)=\mathcal{D}^{p,n}(U,E\otimes A^{\otimes l})$ then $\overline{\partial}f=0$ and $f=\xi$ with constant coefficients on $B_{R+2\varepsilon}$.
    We define $\alpha_m=B^{-1}_{\tilde{\psi}_m}f\in \mathcal{D}^{p,n}(U,E\otimes A^{\otimes l})$. From $i[\Lambda_\omega,\overline{\partial}]=D'^*$ (cf. [Dem10,\,Chapter4]), we get 
    \begin{align*}
        D'\alpha_m=D'B^{-1}_{\tilde{\psi}_m}f&=D'\Bigl(\frac{1}{p}\xi+O(|z|^2)\Bigr)\\
        &=\frac{-1}{p}\xi\wedge\sum\overline{z}_jdz_j+O(|z|^3),\\
        D'^*\alpha_m=D'^*B^{-1}_{\tilde{\psi}_m}f&=D'^*\Bigl(\frac{1}{p}\xi+O(|z|^2)\Bigr)=O(|z|^2)
    \end{align*}
    on $B_R$. And we get 
    \begin{align*}
        D'\alpha_m=D'B^{-1}_{\tilde{\psi}_m}f&=\frac{1}{mp}D'f+O(|z|^2)\\
        &=\frac{1}{mp}\Bigl(\partial f-f\wedge\sum\overline{z}_jdz_j\Bigr)+O(|z|^3),\\
        D'^*\alpha_m=D'^*B^{-1}_{\tilde{\psi}_m}f&=\frac{1}{mp}D'^*f+O(|z|^2)
    \end{align*}
    on $B_{2R}\setminus B_{R+2\varepsilon}$.
    Since $D'\alpha_m(0)=0$ and $\alpha_m$ is constant coefficients on $B_R$, so after shrinking $R$, we can get 
    \begin{align*}
        |D'\alpha_m|^2_{h,\omega}&=\Bigl(\frac{1}{p}\Bigr)^2|\xi\wedge\sum\overline{z}_jdz_j|^2_{h,\omega}\\
        &\leq\Bigl(\frac{1}{p}\Bigr)^2|z|^2|\xi|^2_{h,\omega}\leq \frac{1}{4}c,\\
        |D'^*\alpha_m|^2_{h,\omega}&=|O(|z|^2)|^2_{h,\omega}\leq\frac{1}{4}c
    \end{align*}
    on $B_{R}$. From $f$ and $\alpha_m$ have compact support in $B_{2R}$, there is a constant $C$, such that 
    \begin{align*}
        |\langle[i\Theta_{E,h},\Lambda_\omega]\alpha_m,\alpha_m\rangle_{h,\omega}|&=\Bigl(\frac{1}{mp}\Bigr)^2|\langle[i\Theta_{E,h},\Lambda_\omega]f,f\rangle_{h,\omega}|_{h,\omega}+O(|z|^2)\leq\frac{C}{m^2},\\
        |D'\alpha_m|^2_{h,\omega}&=\Bigl(\frac{1}{mp}\Bigr)^2|D'f|^2_{h,\omega}+O(|z|^2)\\
        &=\Bigl(\frac{1}{mp}\Bigr)^2\bigl(|\partial f|^2_{h,\omega}+m^2|f\wedge\sum z_jdz_j|^2_{h,\omega}\bigr)+O(|z|^2)\\
        &\leq\frac{1}{p^2}\Bigl(\frac{1}{m^2}|\partial f|^2_{h,\omega}+|z|^2|f|^2_{h,\omega}\Bigr)+O(|z|^2)\leq C\Bigl(\frac{1}{m^2}+|z|^2\Bigr),\\
        |D'^*\alpha_m|^2_{h,\omega}&=\Bigl(\frac{1}{mp}\Bigr)^2|D'^*f|^2_{h,\omega}+O(|z|^2)\leq\frac{C}{m^2}
    \end{align*}
    on $B_{2R}\setminus B_{R+2\varepsilon}$ and that 
    \begin{align*}
        |\langle[i\Theta_{E,h},\Lambda_\omega]\alpha_m,\alpha_m\rangle_{h,\omega}|\leq C,\,\,
        |D'\alpha_m|^2_{h,\omega}\leq C,\,\,
        |D'^*\alpha_m|^2_{h,\omega}\leq C
    \end{align*}
    on $B_{R+2\varepsilon}\setminus B_R$ for any $m\geq4$.

    Then we consider the left-hand side of $(\ast)$ with $\alpha$ and $\psi$ replaced by $\alpha_m$ and $\psi_m$ defined as above.
    \begin{align*}
        &\langle\langle[i\Theta_{E,h},\Lambda_\omega]\alpha_m,\alpha_m\rangle\rangle_{\tilde{\psi}_m}+||D'\alpha_m||^2_{\tilde{\psi}_m}+||D'^*\alpha_m||^2_{\tilde{\psi}_m}\\
        &=\int_{B_R}\langle[i\Theta_{E,h},\Lambda_\omega]\alpha_m,\alpha_m\rangle_{h,\omega} e^{-\varphi}dV_\omega+\int_{B_R}\bigl(|D'\alpha_m|_{h,\omega}^2+|D'^*\alpha_m|_{h,\omega}^2\bigr)e^{-\varphi}dV_\omega\\
        &\quad+\int_{B_{R+2\varepsilon}\setminus B_R}\langle[i\Theta_{E,h},\Lambda_\omega]\alpha_m,\alpha_m\rangle_{h,\omega} e^{-\psi_m}dV_\omega+\int_{B_{R+2\varepsilon}\setminus B_R}\bigl(|D'\alpha_m|_{h,\omega}^2+|D'^*\alpha_m|_{h,\omega}^2\bigr)e^{-\psi_m}dV_\omega\\
        &\quad+\int_{B_{2R}\setminus B_{R+2\varepsilon}}\langle[i\Theta_{E,h},\Lambda_\omega]\alpha_m,\alpha_m\rangle_{h,\omega} e^{-m\varphi_\varepsilon}dV_\omega+\int_{B_{2R}\setminus B_{R+2\varepsilon}}\bigl(|D'\alpha_m|_{h,\omega}^2+|D'^*\alpha_m|_{h,\omega}^2\bigr)e^{-m\varphi_\varepsilon}dV_\omega\\
        &\leq-\frac{c}{2}\int_{B_R}e^{-\varphi}dV_\omega+3C\int_{B_{R+2\varepsilon}\setminus B_R}e^{-\psi_m}dV_\omega+C\int_{B_{2R}\setminus B_{R+2\varepsilon}}\Bigl(\frac{3}{m^2}+|z|^2\Bigr)e^{-m\varphi_\varepsilon}dV_\omega\\
        &\leq-\frac{c}{2}\mathrm{Vol}(B_R)+C\Bigl(3\mathrm{Vol}(B_{R+2\varepsilon}\setminus B_R)+\int_{B_{2R}\setminus B_{R+2\varepsilon}}\Bigl(\frac{3}{m^2}+|z|^2\Bigr)e^{-m\varphi_\varepsilon}dV_\omega\Bigr),
    \end{align*}
    where $\varphi<0$ on $B_R$ and $\psi_m,\varphi_\varepsilon>0$ on $B_{2R}\setminus \overline{B}_R$.
    Since $\lim_{m\to+\infty}m\varphi_\varepsilon(z)=+\infty$ for $z\in B_{2R}\setminus \overline{B}_R$, if $4\leq m\to+\infty$ then we get 
    \begin{align*}
        0\leq\int_{B_{2R}\setminus {R+\varepsilon}}\Bigl(\frac{3}{m^2}+|z|^2\Bigr)e^{-m\varphi_\varepsilon}dV_\omega\leq(1+4R^2)\int_{B_{2R}\setminus {R+\varepsilon}}e^{-m\varphi_\varepsilon}dV_\omega\longrightarrow 0.
    \end{align*}
    Here, if $1/2>\varepsilon\to 0$ then
    \begin{align*}
        \mathrm{Vol}(B_{R+2\varepsilon}\setminus B_R)&=\mathrm{Vol}(B_{R+2\varepsilon})-\mathrm{Vol}(B_R)=C_n((R+2\varepsilon)^{2n}-R^{2n})\\
        &=\varepsilon C_n\sum^{2n}_{k=1}
        \begin{pmatrix}
            n \\
            k \\
        \end{pmatrix}
        R^{2n-k}(2\varepsilon)^{k-1}\leq
        2\varepsilon C_n\sum^{2n}_{k=1}
        \begin{pmatrix}
            n \\
            k \\
        \end{pmatrix}
        R^{2n-k}\longrightarrow 0,
    \end{align*}
    where $C_n$ is a constant depending only on $n$.

    Therefore we obtain that 
    \begin{align*}
        \langle\langle[i\Theta_{E,h},\Lambda_\omega]\alpha_m,\alpha_m\rangle\rangle_{\tilde{\psi}_m}+||D'\alpha_m||^2_{\tilde{\psi}_m}+||D'^*\alpha_m||^2_{\tilde{\psi}_m}<0
    \end{align*}
    for $m>>4$ and $1/2>>\varepsilon>0$, which contradicts to the inequality $(\ast)$.
    Hence, we have that $A^{p,n}_{E,h,\omega}\geq0$.

    Finally, we show that $(E,h)$ satisfies the $(p,n)$-$L^2_\omega$-estimate condition on $X$ if $A^{p,n}_{E,h,\omega}\geq0$.
    As above, we have that $[i\Theta_{A,h_A}\otimes\mathrm{id}_E,\Lambda_\omega]>0$ on $\Lambda^{p,n}T^*_X\otimes E\otimes A$.
    Then by the inequality
    \begin{align*}
        A^{p,n}_{E\otimes A,h\otimes h_A,\omega}&=[i\Theta_{E,h},\Lambda_\omega]+[i\Theta_{A,h_A}\otimes\mathrm{id}_E,\Lambda_\omega]\\
        &=A^{p,n}_{E,h,\omega}+[i\Theta_{A,h_A}\otimes\mathrm{id}_E,\Lambda_\omega]\geq [i\Theta_{A,h_A}\otimes\mathrm{id}_E,\Lambda_\omega]>0
    \end{align*}
    and Theorem \ref{L^2-estimate completeness (p,n)-forms}, for any $f\in\mathcal{D}^{p,n}(X,E\otimes A)$ with $\overline{\partial}f=0$, there is $u\in L^2_{p,n-1}(X,E\otimes A)$ satisfying $\overline{\partial}u=f$ and
    \begin{align*}
        \int_X|u|^2_{h\otimes h_A,\omega}dV_\omega\leq\int_X\langle(A^{p,n}_{E\otimes A,h\otimes h_A,\omega})^{-1}f,f\rangle_{h\otimes h_A,\omega}dV_\omega<+\infty.
    \end{align*} 
    Since $A^{p,n}_{E,h,\omega}\geq0$, we have the inequality
    \begin{align*}
        \langle(A^{p,n}_{E\otimes A,h\otimes h_A,\omega})^{-1}f,f\rangle_{h\otimes h_A,\omega}\leq\langle[i\Theta_{A,h_A}\otimes\mathrm{id}_E,\Lambda_\omega]^{-1}f,f\rangle_{h\otimes h_A,\omega}.
    \end{align*}
    Hence, $(E,h)$ satisfies the $(p,n)$-$L^2_\omega$-estimate condition.
\end{proof}

Using the similar proof technique as Theorem \ref{Main thm 2} and [DNWZ20,\,Theorem\,1.1], Theorem \ref{Main thm 1} is proved using the following theorem.
In the case of $(n,q)$-forms, i.e. Theorem \ref{Main thm 1}, the proof itself is easier than in $(p,n)$-forms because $D'\alpha_m$ vanishes.

\begin{theorem}$(\mathrm{cf.~[Dem}$-$\mathrm{book,~ChapterVIII,~Theorem~6.1]})$
    Let $(X,\widehat{\omega})$ be a complete \kah manifold, $\omega$ be another \kah metric which is not necessarily complete and $(E,h)$ be a holomorphic vector bundle which satisfies $A^{n,q}_{E,h,\omega}\geq0$.
    Then for any $\overline{\partial}$-closed $f\in L^2_{n,q}(X,E,h,\omega)$ there exists $u\in L^2_{n,q-1}(X,E,h,\omega)$ satisfies $\overline{\partial}u=f$ and 
    \begin{align*}
        \int_X|u|^2_{h,\omega}dV_{\omega}\leq\int_X\langle(A^{n,q}_{E,h,\omega})^{-1}f,f\rangle_{h,\omega}dV_{\omega},
    \end{align*}
    where we assume that the right-hand side is finite.
\end{theorem}

From Theorem \ref{Main thm 1} and Theorem \ref{A_E and A_E^*}, we obtain the following corollary.

\begin{corollary}\label{(0,q) or (p,0) characterization}
    Let $(X,\omega)$ be a \kah manifold of dimension $n$ which admits a positive holomorphic Hermitian line bundle and $(E,h)$ be a holomorphic Hermitian vector bundle over $X$. Let $p$ and $q$ be positive integers with $q\leq n-1$.
    Then we have the following
    \begin{itemize}
        \item $(E^*,h^*)$ satisfies the $(n,n-q)$-$L^2_\omega$-estimate condition if and only if $A^{0,q}_{E,h,\omega}\geq0$.
        \item $(E^*,h^*)$ satisfies the $(n-p,n)$-$L^2_\omega$-estimate condition if and only if $A^{p,0}_{E,h,\omega}\geq0$.
    \end{itemize} 
\end{corollary}

Here, the characterization of semi-negative curvature operator, i.e. 

$A^{n,q}_{E,h,\omega}\leq0$, $A^{0,q}_{E,h,\omega}\leq0$ for $q\geq0$, $A^{p,n}_{E,h,\omega}\leq0$ for $p\geq1$ and $A^{p,0}_{E,h,\omega}\leq0$ for $n-1\geq p$,\\
by $L^2$-estimates can be obtained immediately by using Theorem \ref{(p,q) and (n-q,n-p)}, Theorem \ref{Main thm 1} and Corollary \ref{(0,q) or (p,0) characterization}.

\section{Characterizations of Nakano semi-negativity}\label{section:4}

In this section, we obtain characterizations of Nakano semi-negativity by $L^2$-estimate, i.e. Theorem \ref{Main thm 2}, from Theorem \ref{Main thm 1} and properties of the curvature operator.
Let $X$ be a Hermitian manifold and $(E,h)$ be a holomorphic Hermitian vector bundle. 
We denote the condition that there exists a Hermitian metric $\omega$ on $X$ such that $A^{p,q}_{E,h,\omega}=[i\Theta_{E,h},\Lambda_\omega]>0$ on $X$ by $A^{p,q}_{E,h}>0$ on $X$.


\begin{lemma}\label{Nakano iff A(n,1)}$(\mathrm{cf.~[DNWZ20,~Lemma ~2.5]})$
    Let $X$ be a complex manifold and $(E,h)$ be a holomorphic Hermitian vector bundle over $X$.
    Then $(E,h)$ is Nakano positive if and only if for any local coordinates $U$, $A^{n,1}_{E,h}>0$ on $U$.
    In particular, 
    if $X$ is a Hermitian manifold then for any Hermitian metric $\omega$ on $X$ we have that
    \begin{align*}
        (E,h)>_{Nak}0\quad \iff\quad  A^{n,1}_{E,h}>0~ on~\forall U \quad \iff\quad  A^{n,1}_{E,h,\omega}>0,
    \end{align*}
    where $U$ is any local coordinates.

    Moreover, we obtain the claim replaced positive with semi-positive or negative, semi-negative respectively.
\end{lemma}

From Theorem \ref{(p,q) and (n-q,n-p)}, Theorem \ref{A_E and A_E^*} and Lemma \ref{Nakano iff A(n,1)}, we obtain the following corollary.

\begin{corollary}\label{Nakano & curvature operator}
    Let $X$ be a complex manifold and $(E,h)$ be a holomorphic Hermitian vector bundle over $X$.
    Then we have that 
    \begin{align*}
        (E,h)>_{Nak}0 & \iff A^{n,1}_{E,h}\,\,\,>0  ~on~U \iff A^{0,n-1}_{E^*,h^*}>0 ~on~U\\
        & \iff A^{n-1,0}_{E,h}\!<0 ~on~U \iff A^{1,n}_{E^*,h^*}<0 ~on~U, ~\, and\\
        (E,h)<_{Nak}0 & \iff A^{n,1}_{E,h}\,\,\,<0  ~on~U \iff A^{0,n-1}_{E^*,h^*}<0 ~on~U\\
        & \iff A^{n-1,0}_{E,h}\!>0 ~on~U \iff A^{1,n}_{E^*,h^*}>0 ~on~U,
    \end{align*}
    where $U$ is any local coordinates. In particular if $X$ is a Hermitian manifold then $U$ can be changed to $X$.
\end{corollary}

Hence, from Theorem \ref{Main thm 2}, Lemma \ref{Nakano iff A(n,1)} and Corollary \ref{Nakano & curvature operator}, we have the following theorem.
This type of theorem for Nakano semi-positive was first shown in \cite{DNWZ20}.

\begin{theorem}$(\mathrm{= Theorem~ \ref{Main thm 3}})$
    Let $(X,\omega)$ be a \kah manifold of dimension $n$ which admits a positive holomorphic Hermitian line bundle and $(E,h)$ be a holomorphic Hermitian vector bundle over $X$. 
    Then $(E^*,h^*)$ satisfies the $(1,n)$-$L^2$-estimate condition on $X$ if and only if $(E,h)$ is Nakano semi-negative.
\end{theorem}

We introduce another notion about Nakano-type positivity.

\begin{definition}$(\mathrm{cf.~}$\cite{LSY13},~\cite{Dem20})
    Let $X$ be a complex manifold of complex dimension $n$ and $(E,h)$ be a holomorphic Hermitian vector bundle of rank $r$ over $X$.
    $(E,h)$ is said to be $\it{dual~Nakano ~positive}$ (resp. $\it{dual~Nakano ~semi}$-$\it{positive}$) if $(E^*,h^*)$ is Nakano negative (resp. Nakano semi-negative).
\end{definition}

From definitions, we see immediately that if $(E,h)$ is Nakano positive or dual Nakano positive then $(E,h)$ is Griffiths positive. And there is an example of dual Nakano positive as follows.
Let $h_{FS}$ be the Fubini-Study metric on $T_{\mathbb{P}^n}$, then $(T_{\mathbb{P}^n},h_{FS})$ is dual Nakano positive and Nakano semi-positive (cf. [LSY13, Corollary 7.3]). 
$(T_{\mathbb{P}^n},h_{FS})$ is easyly shown to be ample, but it is not Nakano positive.
In fact, if $(T_{\mathbb{P}^n},h_{FS})$ is Nakano positive then from the Nakano vanishing theorem (see \cite{Nak55}), we have that 
\begin{align*}
    H^{n-1,n-1}(\mathbb{P}^n,\mathbb{C})=H^{n-1}(\mathbb{P}^n,\Omega^{n-1}_{\mathbb{P}^n})=H^{n-1}(\mathbb{P}^n,K_{\mathbb{P}^n}\otimes T_{\mathbb{P}^n})=0.
\end{align*}
However, this contradicts $H^{n-1,n-1}(\mathbb{P}^n,\mathbb{C})=\mathbb{C}$.

In the same way as in Theorem \ref{Main thm 3}, we obtain the following corollary.

\begin{corollary}
    Let $(X,\omega)$ be a \kah manifold of dimension $n$ which admits a positive holomorphic Hermitian line bundle and $(E,h)$ be a holomorphic Hermitian vector bundle over $X$. 
    Then $(E,h)$ satisfies the $(1,n)$-$L^2$-estimate condition if and only if $(E,h)$ is dual Nakano semi-positive.
\end{corollary}

Finally, we obtain the following characterizations of Nakano semi-negativity, semi-positivity and dual Nakano semi-positive by $L^2$-estimates for smooth Hermitian metrics on a general complex manifold.

\begin{corollary}\label{smooth Nakano semi iff}
    Let $X$ be a complex manifold of dimension $n$ and $(E,h)$ be a holomorphic Hermitian vector bundle over $X$. For any local Stain coordinate system $\{(U_\alpha,\iota_\alpha)\}_\alpha$, we have the following
    \begin{itemize}
        \item $(E,h)$ is Nakano semi-positive if and only if $(E,h)$ satisfies the $(n,1)$-$L^2$-estimate condition on any $U_\alpha$.
        \item $(E,h)$ is Nakano semi-negative if and only if $(E^*,h^*)$ satisfies the $(1,n)$-$L^2$-estimate condition on any $U_\alpha$.
        \item $(E,h)$ is dual Nakano semi-positive if and only if $(E,h)$ satisfies the $(1,n)$-$L^2$-estimate condition on any $U_\alpha$.
    \end{itemize}
\end{corollary}

\begin{proof}
    From Steinness of $U_\alpha$, the set $U_\alpha$ has a complete \kah metric $\widehat{\omega}$.
    Here, the \kah manifold $(U_\alpha,\widehat{\omega})$ admits a positive holomorphic Hermitian line bundle. 
    Since Theorem \ref{Main thm 1} and \ref{Main thm 2} and Corollary \ref{Nakano iff A(n,1)}, this proof is completed.
\end{proof}

\section{Applications}\label{section:5}

In this section, as applications of main theorems, 
we prove that the $(n,q)$ and $(p,n)$-$L^2$-estimate condition is preserved with respect to a increasing sequence.
This phenomenon is first mentioned in \cite{Ina21a} 
as an extension of the properties seen in plurisubharmonic functions. 
After that, it is extended to the case of Nakano semi-positivity in \cite{Ina21b}.

\begin{proposition}\label{sequence of (p,q)-L^2 estimate}
    Let $(X,\omega)$ be a \kah manifold of dimension $n$ which admits a positive holomorphic Hermitian line bundle and $E$ be a holomorphic vector bundle over $X$ equipped with a singular Hermitian metric $h$ and let $p$ be a positive integer.
    Assume that there exists a sequence of smooth Hermitian metrics $\{h_\nu\}_{\nu\in\mathbb{N}}$ increasing to $h$ pointwise such that $A^{p,n}_{E,h_\nu,\omega}\geq0$.
    Then $(E,h)$ satisfies the $(p,n)$-$L^2_\omega$-estimate condition on $X$.

\end{proposition}

\begin{proof}
    For any positive holomorphic Hermitian line bundle $(A,h_A)$ on $X$, for any $f\in\mathcal{D}^{p,n}(X,E\otimes A,h\otimes h_A,\omega)$ with $\overline{\partial}f=0$ we have that $f\in\mathcal{D}^{p,n}(X,E\otimes A,h_\nu\otimes h_A,\omega)$.
    Since $(E,h_\nu)$ satisfies the $(p,n)$-$L^2_\omega$-estimate condition on $X$, we get a solution $u_\nu$ of $\overline{\partial}u_\nu=f$ satisfying
    \begin{align*}
        \int_X|u_\nu|^2_{h_\nu\otimes h_A,\omega}dV_\omega&\leq\int_X\langle[i\Theta_{A,h_A}\otimes\mathrm{id}_E,\Lambda_\omega]^{-1}f,f\rangle_{h_\nu\otimes h_A,\omega}dV_\omega\\
        &\leq\int_X\langle[i\Theta_{A,h_A}\otimes\mathrm{id}_E,\Lambda_\omega]^{-1}f,f\rangle_{h\otimes h_A,\omega}dV_\omega<+\infty
    \end{align*}
    for each $\nu\in\mathbb{N}$. Here, the right-hand side of the inequality above has an upper bound independent of $\nu$. Then $\{u_\nu\}_{\nu\geq j}$ forms a bounded sequence in $L^2_{p,n-1}(X,E\otimes A,h_j\otimes h_A,\omega)$ due to the monotonicity of $\{h_\nu\}$.
    Therefore we can get a weakly convergent subsequence $\{u_{\nu_k}\}_{k\in\mathbb{N}}$ by using a diagonal argument and the monotonicity of $\{h_\nu\}$. We have that $\{u_{\nu_k}\}_{k\in\mathbb{N}}$ weakly converges in $L^2_{p,n-1}(X,E\otimes A,h_\nu\otimes h_A,\omega)$ for every $\nu$.
    Hence, the weak limit denoted by $u_\infty$ satisfies $\overline{\partial}u_\infty=f$ and 
    \begin{align*}
        \int_X|u_\infty|^2_{h\otimes h_A,\omega}dV_\omega\leq\int_X\langle[i\Theta_{A,h_A}\otimes\mathrm{id}_E,\Lambda_\omega]^{-1}f,f\rangle_{h\otimes h_A,\omega}dV_\omega
    \end{align*}
    due to the monotone convergence theorem. From the above, we have that $(E,h)$ satisfies the $(p,n)$-$L^2_\omega$-estimate condition on $X$.
\end{proof}

From Theorem \ref{Main thm 1} and \ref{Main thm 2} and Corollary \ref{(0,q) or (p,0) characterization} and \ref{smooth Nakano semi iff}, it is natural to define semi-positivity of curvature operators, Nakano semi-positive, semi-negative and dual Nakano semi-positive by extending from smooth Hermitian metrics to singular Hermitian metrics as follows.

\begin{definition}\label{sing curvatur operator}
    Let $X$ be a \kah manifold equipped with a complete \kah metric and $E$ be a holomorphic vector bundle over $X$ equipped with a singular Hermitian metric $h$.
    For any positive integers $p$ and $q$ and any \kah metric $\omega$, a curvature operator
    \begin{itemize}
        \item $A^{n,q}_{E,h,\omega}$ is said to be $\it{semi}$-$\it{positive~ in~ the~ sense~ of~ singular}$ 
        if $(E,h)$ satisfies the $(n,q)$-$L^2$-estimate condition on $X$.
        \item $A^{0,q}_{E,h,\omega}$ is said to be $\it{semi}$-$\it{positive~ in~ the~ sense~ of~ singular}$ 
        if $(E^*,h^*)$ satisfies the $(n,n-q)$-$L^2$-estimate condition on $X$, where $q\leq n-1$.
        \item $A^{p,n}_{E,h,\omega}$ is said to be $\it{semi}$-$\it{positive~ in~ the~ sense~ of~ singular}$ 
        if $(E,h)$ satisfies the $(p,n)$-$L^2$-estimate condition on $X$.
        \item $A^{p,0}_{E,h,\omega}$ is said to be $\it{semi}$-$\it{positive~ in~ the~ sense~ of~ singular}$ 
        if $(E^*,h^*)$ satisfies the $(n-p,n)$-$L^2$-estimate condition on $X$.
    \end{itemize}
\end{definition}

\begin{definition}\label{sing Nakano def}
    Let $X$ be a complex manifold and $E$ be a holomorphic vector bundle over $X$ equipped with a singular Hermitian metric $h$. For any local Stain coordinate system $\{(U_\alpha,\iota_\alpha)\}_\alpha$, we define the following.
    \begin{itemize}
        \item $(E,h)$ is said to be $\it{Nakano~ semi}$-$\it{positive~ in~ the~ sense~ of~ singular}$ if $(E,h)$ satisfies the $(n,1)$-$L^2$-estimate condition on any $U_\alpha$.
        \item $(E,h)$ is said to be $\it{Nakano~ semi}$-$\it{negative~ in~ the~ sense~ of~ singular}$ if $(E^*,h^*)$ satisfies the $(1,n)$-$L^2$-estimate condition on any $U_\alpha$.
        \item $(E,h)$ is said to be $\it{dual~ Nakano~ semi}$-$\it{positive~ in~ the~ sense~ of~ singular}$ if $(E,h)$ satisfies the $(1,n)$-$L^2$-estimate condition on any $U_\alpha$.
    \end{itemize}
\end{definition}

Since Proposition \ref{sequence of (p,q)-L^2 estimate} and the proof of [Ina22,\,Proposition\,6.1], we obtain the following corollaries.

\begin{corollary}
    Let $X$ be a \kah manifold equipped with a complete \kah metric and $E$ be a holomorphic vector bundle over $X$ equipped with a (singular) Hermitian metric $h$.
    Let $p$ and $q$ be positive integers and $\omega$ be a \kah metric on $X$. 
    Assume that there exists a sequence of smooth Hermitian metrics $\{h_\nu\}_{\nu\in\mathbb{N}}$ such that $A^{n,q}_{E,h_\nu,\omega}\geq0$ (resp. $A^{p,n}_{E,h_\nu,\omega}\geq0$, $A^{0,q}_{E,h_\nu,\omega}\geq0$, $A^{p,0}_{E,h_\nu,\omega}\geq0$) for each $\nu$.

    If $\{h_\nu\}_{\nu\in\mathbb{N}}$ increases to $h$ pointwise then
    $A^{n,q}_{E,h,\omega}$ (resp. $A^{p,n}_{E,h,\omega}$) is semi-positive in the sense of singular as in Definition \ref{sing curvatur operator}.

    If $\{h_\nu\}_{\nu\in\mathbb{N}}$ decreases to $h$ pointwise then
    $A^{0,q}_{E,h,\omega}$ (resp. $A^{p,0}_{E,h,\omega}$) is semi-positive in the sense of singular as in Definition \ref{sing curvatur operator}.
\end{corollary}

\begin{corollary}$(\mathrm{cf.~[Ina22,~Proposition ~6.1]})$
    Let $X$ be a complex manifold and $E$ be a holomorphic vector bundle over $X$ equipped with a singular Hermitian metric $h$.
    Assume that there exists a sequence of smooth Nakano semi-positive metrics $\{h_\nu\}_{\nu\in\mathbb{N}}$ increasing to $h$ pointwise.
    Then $h$ is Nakano semi-positive in the sense of singular as in Definition \ref{sing Nakano def}.
\end{corollary}

\begin{corollary}
    Let $X$ be a complex manifold and $E$ be a holomorphic vector bundle over $X$ equipped with a singular Hermitian metric $h$.
    Assume that there exists a sequence of smooth Nakano semi-negative metrics $\{h_\nu\}_{\nu\in\mathbb{N}}$ decreasing to $h$ pointwise.
    Then $h$ is Nakano semi-negative in the sense of singular as in Definition \ref{sing Nakano def}.
\end{corollary}

\begin{corollary}
    Let $X$ be a complex manifold and $E$ be a holomorphic vector bundle over $X$ equipped with a singular Hermitian metric $h$.
    Assume that there exists a sequence of smooth dual Nakano semi-positive metrics $\{h_\nu\}_{\nu\in\mathbb{N}}$ increasing to $h$ pointwise.
    Then $h$ is dual Nakano semi-positive in the sense of singular as in Definition \ref{sing Nakano def}.
\end{corollary}

{\bf Acknowledgement. } 
I would like to thank my supervisor Professor Shigeharu Takayama for guidance and helpful advice. I would also like to thank Professor Takahiro Inayama for useful advice on applications.

$ $

\rightline{\begin{tabular}{c}
    $\it{Yuta~Watanabe}$ \\
    $\it{Graduate~School~of~Mathematical~Sciences}$ \\
    $\it{The~University~of~Tokyo}$ \\
    $3$-$8$-$1$ $\it{Komaba, ~Meguro}$-$\it{ku}$ \\
    $\it{Tokyo, ~Japan}$ \\
    ($E$-$mail$ $address$: watayu@g.ecc.u-tokyo.ac.jp)
\end{tabular}}

\end{document}